\documentclass{amsart}
\usepackage{amsrefs,amssymb,delimset,siunitx,booktabs,tikz,tkz-berge}
\SetVertexNoLabel
\tikzstyle{EdgeStyle}=[line width=1.2pt, color=cyan]
\tikzstyle{VertexStyle}=[shape=circle, inner sep=6pt, ball color=green!20]

\usepackage[colorlinks, citecolor=cyan]{hyperref}
\newcolumntype{L}{>{$}l<{$}} % math-mode version of "l" column type 
\newcolumntype{R}{>{$}r<{$}} % math-mode version of "r" column type 
\newcolumntype{C}{>{$}c<{$}} % math-mode version of "l" column type
\numberwithin{equation}{section}
\numberwithin{table}{section}
\numberwithin{figure}{section}

\newcommand\floor[1]{\lfloor#1\rfloor} 
\newcommand\ceil[1]{\lceil#1\rceil}

\usepackage[capitalize]{cleveref}
\crefname{itm}{}{}
\creflabelformat{itm}{~\upshape(#2#1#3)}
\crefname{ineq}{Ineq.}{Ineqs.}
\creflabelformat{ineq}{~\upshape(#2#1#3)}
\crefname{sec}{\S}{\S}

\newtheorem{theorem}{Theorem}[section]

\newtheorem{lemma}[theorem]{Lemma}
\newtheorem{proposition}[theorem]{Proposition}

\usepackage{enumitem}
\setlist[enumerate,itemize]{
leftmargin=0.6cm, 
labelsep=5pt, 
}

\allowbreak
\allowdisplaybreaks

\title[The edge dimension of $P(n,3)$ is 4]{{
The edge dimension of the generalized Petersen graph $P(n,3)$ is 4}}

\author[D.G.L. Wang]{David G.L. Wang$^\dag$$^\ddag$}
\address{
$^\dag$School of Mathematics and Statistics, Beijing Institute of 
Technology, 102488 Beijing, P. R. China\\
$^\ddag$Beijing Key Laboratory on MCAACI, Beijing Institute of 
Technology, 102488 Beijing, P. R. China}
\email{glw@bit.edu.cn}

\author[M.M.Y. Wang]{Monica M.Y. Wang}
\address{
School of Mathematics and Statistics, Beijing Institute of 
Technology, 102488 Beijing, P. R. China}
\email{mengyu919@bit.edu.cn}

\author[S.Q. Zhang]{Shiqiang Zhang}
\address{
School of Mathematics and Statistics, Beijing Institute of 
Technology, 102488 Beijing, P. R. China}
\email{shiqiang@bit.edu.cn}

\thanks{Corresponding author: David G.L. Wang.}
\thanks{This paper was supported by National Natural Science Foundation of China (Grant No.\ 11671037).}

\begin{document}

\begin{abstract}
It is known that the problem of computing the edge dimension of a graph is NP-hard, and that the edge dimension of any generalized Petersen graph $P(n,k)$ is at least 3. We prove that the graph $P(n,3)$ has edge dimension 4 for $n\ge 11$, by showing semi-combinatorially the nonexistence of an edge resolving set of order 3 and by constructing explicitly an edge resolving set of order 4.
\end{abstract}

\maketitle                   
\tableofcontents

\section{Introduction}
Let $n\ge 3$ and $1\le k<n/2$.
The \emph{generalized Petersen graph}, denoted $P(n,k)$, 
is the graph with vertex set
$\{u_j,v_j\colon j\in \mathbb{Z}_n\}$
and edge set 
$\{u_ju_{j+1},\,v_jv_{j+k},\,u_jv_j\,|\, j\in\mathbb{Z}_n\}$,
where $\mathbb{Z}_n$ is the additive group 
of integers modulo $n$.
The generalized Petersen graphs
was introduced in 1950 by Coxeter~\cite{Cox50}
and was given its name in 1969 by Watkins~\cite{Wat69}
in a consideration of a conjecture of Tutte~\cite{Tut69B}.
For extensive surveys on the Petersen graph $P(5,2)$,
see \cite{CHW92,HS93B}, which also involve all kinds of variantions 
of the Petersen graph, including the family $P(n,k)$.
Many structural and algorithmic properties of the generalized Petersen graphs have been extensively investigated,
among which \cite{Als83,ARR81,Ban78,BBP08,BPZ05,CP72,FGW71,Lov97,NS95,Sch89,SR84}
are oft-cited work that are used to introduce the family $P(n,k)$.
Recent advances in the field include \cite{XK11,DM17,EG19,JW19,KMM17,YW18,KW20+}, 
most of which are concerned with some specific graph theoretical concept for $P(n,k)$.

For NP-hard and NP-complete problems, 
research on the graph $P(n,3)$ often attracts attention
and usually costs a considerable effort. 
For instance, Hlin\v en\'y~\cite{Hli06} showed 
that deteriming the crossing number of a cubic graph is NP-hard;
Richter and Salazar~\cite{RS02} found a formula for the crossing number for $P(n,3)$
according to the residuce of $n$ modulo 3.
As another example,
Bre\v sara and \v Sumenjakb~\cite{BS07}
showed the NP-completeness of the 2-rainbow domination problem;
Xu~\cite{Xu09} enhanced the upper bound of the 2-rainbow domination number of $P(n,3)$ according to the residuce of $n$ modulo~16. 

In 2018,
Kelenc, Tratnik, and Yero~\cite{KTY18} 
introduced the concept of \emph{edge dimension} for a graph,
and showed the NP-hardness of computing the edge dimension of a graph.
They also pointed out that the edge dimension has applications in network security surveillance. In fact, an intruder accesses a network through edges
can be identified by an edge resolving set. 
This paper is concerned with the edge dimension of the graph $P(n,3)$.

For any list $w=(v_1,v_2,\dots,v_k)$ 
of vertices and any vertex $v$ in a connected graph $G$, 
the \emph{representation} of $v$ with respect to $w$ is the list
$\brk1{d(v,v_1),d(v,v_2),\dots,d(v,v_k)}$, 
where $d(x,y)$ is the distance between the vertices $x$ and $y$. 
The set $\{v_1,v_2,\dots,v_k\}$ is said to be 
a \emph{resolving set} for~$G$
if every two vertices of $G$ have distinct representations. 
It was Slater \cite{Sla75} who firstly considered the minimum cardinality of a resolving set for $G$, called the \emph{metric dimension} of $G$. 
In analog, the \emph{edge dimension} for a graph
is the minimum cardinality of a vertex set $\{v_1,v_2,\dots,v_k\}$ 
such that the lists $\brk1{d(e,v_1),d(e,v_2),\dots,d(e,v_k)}$ 
for all edges $e$ are distinct,
where $d(e,v_i)$ is the distance between the edge $e$ and the vertex $v_i$.
For instance, the graph $P(8,3)$ with an edge resolving set $\{u_0,u_1,u_2,v_3\}$
is illustrated in \cref{fig:P83}.
\begin{figure}[htbp]
\begin{tikzpicture}[scale=0.7, rotate=90]
\grCycle[prefix=u,RA=3]{8}
\AssignVertexLabel{u}{$u_0$, $u_7$, $u_6$, $u_5$, $u_4$, $u_3$, $u_2$, $u_1$};
\grCirculant[prefix=v,RA=1.5]{8}{3}
\AssignVertexLabel{v}{$v_0$, $v_7$, $v_6$, $v_5$, $v_4$, $v_3$, $v_2$, $v_1$};
\EdgeIdentity{u}{v}{8}
\draw[color=red, thick] (3,0) circle[radius=.6cm];
\draw[color=red, thick] (2.15,-2.12) circle[radius=.6cm];
\draw[color=red, thick] (0,-3) circle[radius=.6cm];
\draw[color=red, thick] (-1.07,-1.07) circle[radius=.6cm];
\end{tikzpicture}
\caption{The generalized Petersen graph $P(8,3)$ with an edge resolving set $\{u_0,u_1,u_2,v_3\}$.}\label{fig:P83}
\end{figure}
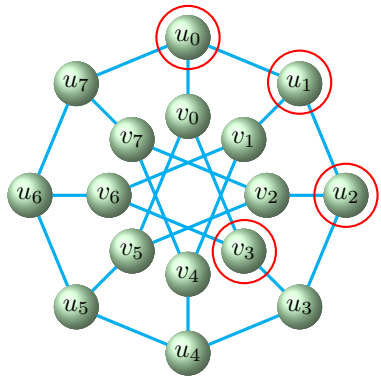
The exact values of the edge dimension for some classes of graphs were known, 
while bounds are given to some other graph classes.
The edge dimension of an $n$-vertex graph is at most $n-1$; see~\cite{Zub18,ZTSX19}.
For a rich resource of different kinds of resolving sets of graphs with applications,
see Kelenc, Kuziak, Taranenko and Yero~\cite{KKTY17}.
More progress on the edge dimension 
can be found from~\cite{Yero16,PY20}. 

Fillipovi\'c, Kartelj and Kratica~\cite{FKK19} 
showed that the edge dimension of the generalized Petersen graph $P(n,k)$ is at least~3, and confirmed that the edge dimension of $P(n,k)$ equals 3 for $k\in\{1,2\}$ and $n\ge 10$. 
Here is our main result.

\begin{theorem}\label{thm:edim=4}
For $n\ge 11$, 
the edge dimension of the generalized Petersen graph $P(n,3)$
is 4.
\end{theorem}

This paper is organized as follows. 
In \cref{sec:fml:dist}, we give a formula for the distance between any vertex and any edge in the graph $P(n,3)$, which serves as the basis for all arguments in the sequel. 
In \cref{sec:edim>=4,sec:edim<=4}, we prove that the edge dimension of $P(n,3)$
has lower and upper bound 4, respectiely. Moreover, the vertex tetrad $\{u_0,u_1,x,y\}$ with $\{x,y\}$ presented in \cref{tab:ResolvingTetrad}
is an edge resolving set of $P(n,3)$ for $n\ge 19$.
\begin{table}[h]
\centering
\caption{Vertex pairs $\{x,y\}$ such that the tetrad
$\{u_0,u_1,x,y\}$ is an edge resolving set of the graph $P(n,3)$, 
where $n\ge 19$, and $r_n'$ is the residue of $n$ modulo 6.}
\label{tab:ResolvingTetrad}
\begin{tabular}{@{} CCCCC @{}}
\addlinespace
\toprule
r_n' & 0,1,3 & 2 & 4 & 5\\
\midrule
\{x,y\}
& \brk[c]1{v_2,\,u_{\floor{n/2}-1}} 
& \brk[c]1{u_{n/2-3},\,v_{n/2-2}} 
& \brk[c]1{v_2,\,u_{n/2+3}} 
& \brk[c]1{u_{\floor{n/2}-1},\,v_{\floor{n/2}}} \\
\bottomrule
\end{tabular}
\end{table}

\section{Strategy of our proof}\label[sec]{sec:strategy}
The edge dimension is all about the distance between a vertex and an edge. 
It is known that the time complexity of finding the shortest paths 
between any two given vertices in a graph is $O(n^3)$;
see the Floyd-Warshall algorithm in \cite[\S 25.2]{CLRS09}. 
In order to establish \cref{thm:edim=4}, however, 
we further need a formula for the distance between a given vertex and a given edge.
Denote
\[
e_i^u=u_i u_{i+1},\quad
e_i^v=v_i v_{i+3},\quad\text{and}\quad
e_i^s=u_iv_i.
\]
For $n\in\mathbb{Z}$,
define $q_n, r_n\in\mathbb{Z}$ by $n=3q_n+r_n$, where $r_n\in\{0,1,2\}$. 

\begin{theorem}\label{thm:dist:ve}
Let $n\ge 13$ and $0\le i\le n-1$. In the graph $P(n,3)$, we have
\begin{align*}
d(u_0,\,e_i^u)
&=\begin{cases}
\min(i,\,n-1-i),&\text{if $i\le 2$ or $i\ge n-3$},\\
\min\brk1{\ceil{i/3}+2,\,\ceil{(n-i-1)/3}+2},&\text{otherwise};
\end{cases}\\%%%%%%%%%%%%%%%%%%%%%%%%%%%%%%%%%%%%%%%%%%%%%%%%%%%%%%%%%%%%uo
d(u_0,\,e_i^s)
&=\begin{cases}
\min(i,\,n-i),&\text{if $i\le 2$ or $i\ge n-2$},\\
\min\brk1{q_i+r_i+1,\,q_n-q_i+r+1},&\text{otherwise};
%m_{q_n+r_i+r+2}(q_i+r_i+1),
\end{cases}\\%%%%%%%%%%%%%%%%%%%%%%%%%%%%%%%%%%%%%%%%%%%%%%%%%%%%%%%%%%%%us
d(u_0,\,e_i^v)
&=\begin{cases}
\min(i+1,\,n-i+1),&\text{if $i\le 1$ or $i=n-1$},\\
\min(q_i+r_i+1,\,q_n-q_i+r),&\text{otherwise};
\end{cases}\\%%%%%%%%%%%%%%%%%%%%%%%%%%%%%%%%%%%%%%%%%%%%%%%%%%%%%%%%%%%%%uv
d(v_0,\,e_i^u)
&=\begin{cases}
\min(q_i+\floor{r_i/2}+1,\,q_n-q_i+\floor{r_n/2}-\floor{r_i/2}+1),
%m_{q_n+\floor{r_n/2}+2}(q_i+\floor{r_i/2}+1),
&\text{if $r_i=0$ or $(r_n,\,r_i)=(0,2)$},\\
\min(q_i+2,\,q_n-q_i+1),&\text{otherwise};
%m_{q_n+3}(q_i+2),
\end{cases}\\%%%%%%%%%%%%%%%%%%%%%%%%%%%%%%%%%%%%%%%%%%%%%%%%%%%%%%%%%%%%%%vo
d(v_0,\,e_i^s)
&=\begin{cases}
\min(q_i,\,q_n-q_i),&\text{if $r_i=r_n=0$},\\
%m_{q_n}(q_i),&\text{if $r_i=r_n=0$},\\
\min(q_i+r_n+1,\,q_n-q_i),&\text{if $r_i=r_n\ne0$},\\
%m_{q_n+r_n+1}(q_i+r_n+1)
\min(q_i,\,q_n-q_i+r_n+1),&\text{if $r_i=0\ne r_n$},\\
%m_{q_n+r_n+1}(q_i),&\text{if $r_i=0\ne r_n$},\\
\min(q_i+2,\,q_n-q_i+2),&\text{if $r_i=1\ne r_n$},\\
%m_{q_n}(q_i)+2,&\text{if $r_i=1\ne r_n$},\\
\min(q_i+3,\,q_n-q_i+r_n),&\text{if $r_i=2\ne r_n$};
%m_{q_n+r_n+2}(q_i+2)+1,&\text{if $r_i=2\ne r_n$};
\end{cases}\\%%%%%%%%%%%%%%%%%%%%%%%%%%%%%%%%%%%%%%%%%%%%%%%%%%%%%%%%%%%%%%vs
d(v_0,\,e_i^v)
&=\begin{cases}
\min(q_i,\,q_n-q_i-1),&\text{if $r_i=r_n=0$},\\
\min(q_i+r_i+2,\,q_n-q_i+r_n-r_i-1),&\text{if $r_i=r_n\ne0$},\\
\min(q_i,\,q_n-q_i+r_n+1),&\text{if $r_i=0\ne r_n$},\\
\min(q_i+3,\,q_n-q_i+2),&\text{if $r_i=1\ne r_n$},\\
\min(q_i+4,\,q_n-q_i+r_n+1),&\text{if $r_i=2\ne r_n$}.
%m_{q_n-1}(q_i),&\text{if $r_i=r_n=0$},\\
%m_{q_n+r_n+1}(q_i+r_i+2),&\text{if $r_i=r_n\ne0$},\\
%m_{q_n+r_n+1}(q_i),&\text{if $r_i=0\ne r_n$},\\
%m_{q_n+1}(q_i+1)+2,&\text{if $r_i=1\ne r_n$},\\
%m_{q_n+r_n+3}(q_i+3)+1,&\text{if $r_i=2\ne r_n$}.
\end{cases}%%%%%%%%%%%%%%%%%%%%%%%%%%%%%%%%%%%%%%%%%%%%%%%%%%%%%%%%%%%%%%vi
\end{align*}
Here $r=\abs{r_n-r_i}$, except when $(r_n,r_i)=(0,2)$ and $r=0$.
\end{theorem}

In \cref{sec:fml:dist}, we will establish \cref{thm:dist:ve} in 
the following 3 steps.
\begin{description}
\item[Step 1]
Figure out the structure of a shortest path between 2 vertices 
and obtain \cref{thm:undeviating}.
\item[Step 2]
Find a formula between any 2 vertices in $P(n,3)$;
see \cref{thm:dist:vv:i<n/2}.
\item[Step 3]
Compare the distances $d(x,y)$ and $d(x,z)$ 
for any vertex $x$ and any edge $yz$ in $P(n,3)$, and obtain the desired 
formula in \cref{thm:dist:ve}.
\end{description}
\Cref{thm:dist:ve} will serve as the basis of verifying and computing 
the distance between a vertex and an edge in $P(n,3)$.

Now we clarify the strategy of proving \cref{thm:edim=4}.
We use the word \emph{triad} (resp., \emph{tetrad})
to denote a set of order 3 (resp., 4),
and the word \emph{triple} (resp., \emph{quadruple})
to mean an ordered list of 3 (resp., 4) elements.
\Cref{thm:edim=4} is an immediate consequence of the combination of \cref{prop:edim>=4,prop:edim<=4}.

\begin{proposition}\label{prop:edim>=4}
For $n\ge 11$, 
the graph $P(n,3)$ has no edge resolving triads.
\end{proposition}

\begin{proposition}\label{prop:edim<=4}
For $n\ge 11$, the graph $P(n,3)$ has an edge resolving tetrad.
\end{proposition}

First of all, we point out that it suffices to show \cref{prop:edim>=4,prop:edim<=4} for $n\ge 100$.
In fact, for $n<100$,
one may quickly check that every vertex triad is not an edge resolving set
with the aid of a computer verification.
Precisely speaking,
we can suppose that $R=\{\alpha_0,\,\beta_y,\,\gamma_z\}$ by symmetry,
where $\alpha,\beta,\gamma\in\{u,v\}$ and $y,z\in\mathbb{Z}_n$.
With \cref{thm:dist:ve} in hand,
the time complexity of
checking that any vertex triad is not an edge resolving set is $O(n^4)$,
since there are $O(n^2)$ choices for $R$ and $O(n^2)$ choices for an edge pair.
In practice,
on a laptop with CPU clock rate \SI{2.3}{\giga\hertz} and 8 kernels,
it costs less than 3 minutes to confirm \cref{prop:edim>=4} for all $11\le n<100$. 
Similarly, the time complexity of checking that every tetrad is or is not 
an ecge resolving set is $O(n^5)$.
In practice, with the aid of \cref{tab:ResolvingTetrad},
verifying \cref{prop:edim<=4} for $n<100$ costs less than 3 hours.

From now on, we can suppose that $n\ge 100$.
This assumption saves us from considering an annoying number of singular cases 
appeared for small $n$. 
Though it can be made even earlier before proving 
\cref{thm:dist:ve},
our proof of \cref{thm:dist:ve} works well for $n<100$.

The strategy of proving \cref{prop:edim>=4} is as follows.
First, we concentrate on the pairs of adjacent edges on the outer cycle and
determine the sets 
\[
A_0=\{i\in\mathbb{Z}_n\colon d(u_0,e_{i-1}^u)=d(u_0,e_i^u)\}
\quad\text{and}\quad
D_0=\{i\in\mathbb{Z}_n\colon d(v_0,e_{i-1}^u)=d(v_0,e_i^u)\}.
\]
As will be seen, it turns out quite fortunately that $A_0\subseteq D_0$.
Second, 
using the derived expressions of~$A_0$ and $D_0$ and by a symmetry argument, 
we can reduce \cref{prop:edim>=4} to
that for each pair $(a,b)\in S_n$, where
\[
S_n=\brk[c]1{(a,b)\in\mathbb{Z}^2\colon
1\le a\le \floor{n/3},\
2a\le b\le \floor{(n+a)/2}},
\]
no triad in the collection
\[
T(a,b)=\brk[c]1{\{\alpha_0,\,\beta_a,\,\gamma_b\}\colon \alpha,\beta,\gamma\in\{u,v\}}
\]
is an edge resolving set.
Since $\abs{T(a,b)}=8$, it is plausible
that the number of triads that we have to consider is still considerably large.
Born under a lucky star, in the third step, 
we manage to prove that $A_0\cap A_a\cap A_b\ne\emptyset$
for almost all pairs $(a,b)\in S_n$, where $A_t=\{x+t\colon x\in A_0\}$;
see \cref{sec:S-W}.
This leads us to a collection of 22 sporadic cases of pairs $(a,b)$, for which
we handle in \cref{sec:W}. 

We prove \cref{prop:edim>=4} according to the residuce $r_n'$ of $n$ modulo 6.
We adopt the same strategy to deal with each of the residues, which is as follows.
First, we show that the edges whose distances from~$u_0$ are less than $3$ are distinguishable. To do this, we list all such edges $e_1$ and compute 
the distances $d(v,e_1)$ for each of the prescribed vertex $v\in\{u_0,v_0,x,y\}$;
see \cref{tab:ResolvingTetrad}. This step is easy to be done by using \cref{thm:dist:ve}.
Second, for any fixed integer $d\ge 3$,
we find out all edges $e_2$ whose distances from $u_0$ are $d$.
The number of such edges $e_2$ for each residue case is about 20.
Computing the distances $d(v,e_2)$ for each prescibed vertices $v\in\{v_0,x,y\}$,
we obtain a table, from which one may see that the edge metric representations
with respect to the prescribed vertex tetrad are distinct.

\section{A formula for the vertex-edge distances in $P(n,3)$}\label[sec]{sec:fml:dist}

Let $G=(V,E)$ be a connected graph.
The distance between a vertex $x\in V$ 
and an edge $e=uv\in E$ 
is the integer $d(x,e)=\min\{d(x,u),\,d(x,v)\}$,
where $d(x,y)$ is the distance between $x$ and $y$.
We say that $x$ resolves two edges $e$ and $e'$ if $d(x,e)\neq d(x,e')$.
A vertex set $R=\{x_1, x_2, \dots, x_k\}$
is said to be an \emph{edge resolving set}
if every two distinct edges in $G$ are resolved by a member of $R$.
The \emph{edge dimension} for $G$
is the minimum cardinality of an edge resolving set of $G$.
For any edge $e\in E$, the list $\brk1{d(e,x_1),d(e,x_2),\dots,d(e,x_k)}$ 
is called the \emph{edge metric representation} of $e$ with respect to $R$.

In the generalized Petersen graph, 
the edges $u_iv_i$ are said to be \emph{spokes},
the subgraph induced by the vertices $u_0,u_1,\dots,u_{n-1}$ 
is said to be the \emph{outer cycle},
and that induced by the vertices $v_0,v_1,\dots,v_{n-1}$
is said to be the \emph{inner cycle(s)}.

\begin{lemma}\label{lem:dist:symm}
Consider the graph $P(n,k)$.
Let $i\in\mathbb{Z}$, $p_0\in\{u_0,\,v_0\}$ and $w\in\{u,\,v\}$. Then 
\[
d(p_0,\,w_i)=d(p_0,\,w_{-i})
\quad\text{and}\quad
d(u_0,\,v_i)=d(v_0,\,u_i). 
\]
Moreover, for any $i,j\in\mathbb{Z}$, we have
\begin{align*}
d(u_i,\,e_j^u)&=d(u_0,\,e_{j-i}^u)=d(u_0,\,e_{i-j-1}^u),\\
d(u_i,\,e_j^s)&=d(u_0,\,e_{j-i}^s)=d(u_0,\,e_{i-j}^s),\\
d(u_i,\,e_j^v)&=d(u_0,\,e_{j-i}^v)=d(u_0,\,e_{i-j-k}^v).
\end{align*}
\end{lemma}
\begin{proof}
All the results follow from the symmetry of the graph $P(n,k)$.
\end{proof}

Let $p=w_0 w_1 \dotsm w_l$ be a path in $P(n,k)$,
where $w_i$ are vertices of $P(n,k)$. Denote the length $l$ by~$\ell_p$.
We can rewrite $p=p_0p_1\dotsm p_s$, 
where each $p_j$ is a path contained entirely in either the outer cycle
or an inner cycle.
Accordingly we call $p_j$ an \emph{outer section} or \emph{inner section}.
We say that an edge $e$ is \emph{clockwise} 
if $e$ is of the form $u_j u_{j+1}$ or $v_j v_{j+k}$,
and \emph{counterclockwise} otherwise.
%Note that the length of a section~$p_j$ is positive except when
%either (i) $j=0$ and $w_0w_1$ is a spoke,
%or (ii) $j=s$ and $w_{l-1}w_l$ is a spoke.
A section $p_j$ is \emph{clockwise}
(resp., \emph{counterclockwise}) if every edge in $p_j$ is so.
The path $p$ \emph{clockwise} (resp., \emph{counterclockwise}) 
if every section $p_j$ is so.
We call $p$ \emph{undeviating} if it is either clockwise or counterclockwise.
For any $1\le j\le s$, we call the subpath $p_{j-1}p_{j}$ 
a \emph{turn} of $p$ if one of the subpaths $p_{j-1}$ and $p_j$
is clockwise and the other is counterclockwise.

\begin{theorem}\label{thm:undeviating}
For any vertices $x$ and $y$ in the grpah $P(n,3)$,
there is a shortest path from $x$ to $y$
which is undeviating and contains at most 2 spokes.
\end{theorem}
\begin{proof}
Without loss of generality we can suppose that $x\in\{u_0,v_0\}$ and $y=w_t$,
where $w\in\{u,v\}$ and $0\le t\le n-1$.
Let $p$ be a path from $x$ to $y$ of length 
$\ell_p = d(x,y)$,
with the minimum number of turns.
If there is no undeviating path from $x$ to $y$,
then $p$ has at least one turn.
We will show that this is impossible by contradiction.
Let $qq'$ be the first turn in $p$, 
where $q$ and $q'$ are sections of $p$.

\noindent\emph{Case 1. $q$ is a clockwise outer section.}
Then we can suppose that $q=u_{j-s}u_{j-s+1}\dotsm u_j$ 
for some $s\ge 1$ and $p=\alpha q v_j \beta$,
where $\alpha$ is a clockwise subpath and $\beta$ is a path starting from $v_{j-3}$.

If $s\ge 2$, then the path
\[
p'=\begin{cases}
\alpha u_{j-s} u_{j-s+1}\dotsm u_{j-3} \beta,&\text{if $s\ge 3$}\\
\alpha u_{j-3}\beta,&\text{if $s=2$}
\end{cases}
\]
is shorter than $p$, contradicting the choice of $p$.
When $s=1$, the path $p$ reduces to $p=\alpha u_{j-1}u_j v_j \beta$.
\begin{itemize}
\item
If $\alpha$ contains a non-spoke arc, 
then $\alpha=\alpha' v_{j-4}v_{j-1}$ for some clockwise path $\alpha'$,
and the path $\alpha' v_{j-4}u_{j-4}u_{j-3}\beta$ is shorter than $p$,
the same contradiction.
\item
If $\alpha$ has no non-spoke arcs,
then the path $\alpha u_{j-1} u_{j-2} u_{j-3} \beta$ has
the same length as $p$
and a less number of turns than~$p$.
\end{itemize}
This proves that Case 1 is impossible.

\noindent\emph{Case 2. $q$ is a counterclockwise outer section.}
We define a reflection $f$ on the vertex set of $P(n,k)$ by $f(w_t)=w_{-t}$,
where $w\in\{u,v\}$. It extends naturally to act on paths.
The path $f(p)$, which is from $f(x)$ to $f(y)$,
has length $\ell_{f(p)}=d(f(x),\, f(y))$ and the minimum number of turns.
Since $f(q)f(q')$ is the first turn in $f(p)$, and $f(q)$ is a clockwise outer section, we know that Case 2 is impossible by the impossibility of Case 1.

\noindent\emph{Case 3. $q'$ is a counterclockwise outer section.}
Then we can suppose that $q'=u_j u_{j-1} \dotsm u_{j-s}$ and
$p=\alpha v_j q' \beta$,
where $s\ge 1$, $\alpha$ is a clockwise subpath ending at $v_{j-3}$,
and $\beta$ is a subpath.

If $s\ge 2$, then the path 
\[
p'=\begin{cases}
\alpha u_{j-3}u_{j-4}\dotsm u_{j-s}\beta,&\text{if $s\ge 3$}\\
\alpha u_{j-3}u_{j-2}\beta,&\text{if $s=2$}
\end{cases}
\]
is shorter than $p$, contradicting the choice of $p$.
When $s=1$, $p$ reduces to $p=\alpha v_j u_j u_{j-1} \beta$.
\begin{itemize}
\item
If $\beta$ contains no non-spoke arcs, 
then the path $\alpha u_{j-3}u_{j-2}u_{j-1}\beta$
has the same length as $p$ and a less number of turns than $p$.
\item
If $\beta$ contains a non-spoke arc,
then $\beta=v_{j-1}\beta'$ for some path $\beta'$ starting from $v_{j-4}$ or from $v_{j+2}$.
In the former case, the path $\alpha u_{j-3}u_{j-4}\beta'$
is shorter than $p$; in the latter case, the path $\alpha v_j u_j u_{j+1} u_{j+2} \beta'$ has the same length as $p$ and a less number of turns than $p$.
\end{itemize}
This proves that Case 3 is impossible.

\noindent\emph{Case 4. $q'$ is a clockwise outer section.}
Since $f(q)f(q')$ is the first turn in the path $f(p)$, and $f(q')$ is a counterclockwise outer section, we know that Case 4 is impossible by the impossibility of Case 3.

This proves the existence of an undeviating path from $x$ to $y$ of length $d(x,y)$.
Let $p$ be such a path.
It remains to show that the number of spokes in $p$ is at most 2. By symmetry, we can suppose that $p$ is clockwise. Assume that $p$ has at least 3 spokes.

If $x=v_0$, then there exists $h\ge 0$ and $i,j\ge 1$ such that
\[
p=v_0v_3\dotsm v_{3h} 
u_{3h} u_{3h+1} \dotsm u_{3h+i}
v_{3h+i} v_{3h+i+3} \dotsm v_{3h+i+3j}
\beta,
\]
where $\beta$ is a clockwise path starting from $u_{3h+i+3j}$. 
Then
\[
\ell_p=h+1+i+1+j+1+\ell_\beta=h+i+j+\ell_\beta+3.
\]
Suppose that $3h+i+3j=3q+r$, where $q\ge 1$ and $r\in\{0,1,2\}$. 
The path 
\[
p'=v_0 v_3 \dotsm v_{3q} u_{3q} u_{3q+1} \dotsm u_{3q+r-1}\beta.
\]
has length $\ell_{p'}=q+r+1+\ell_\beta$, and
\[
\ell_p-\ell_{p'}
=2+(i-r)+(h+j-q)
=2\brk2{\frac{i}{3}+1-\frac{r}{3}}
\ge 2\brk2{\frac{1}{3}+1-\frac{2}{3}}
=\frac{4}{3},
\]
contradicting the choice of $p$. 

Otherwise $x=u_0$. Then there exists $h\ge 0$ and $i,j\ge 1$ such that
\[
p=u_0u_1\dotsm u_h v_h v_{h+3} \dotsm v_{h+3i} u_{h+3i} u_{h+3i+1} \dotsm u_{h+3i+j}\beta,
\]
where $\beta$ is a clockwise path starting from $v_{h+3i+j}$.
Then
\[
\ell_p=h+1+i+1+j+1+\ell_\beta=h+i+j+\ell_\beta+3.
\]
Suppose that $h+3i+j=3q+r$, where $q\ge 1$ and $r\in\{0,1,2\}$. 
The path 
\[
p'=v_0 v_3 \dotsm v_{3q} u_{3q} u_{3q+1} \dotsm u_{3q+r-1}\beta.
\]
has length $\ell_{p'}=q+r+1+\ell_\beta$, and
\[
\ell_p-\ell_{p'}=2+(i-q)+(h+j-r)
=2\brk2{\frac{h+j}{3}+1-\frac{r}{3}}
\ge 2\brk2{\frac{1}{3}+1-\frac{2}{3}}=\frac{4}{3},
\]
contradicting the choice of $p$. This completes the proof.
\end{proof}

Now we enter Step 2.
Let $x$ be a vertex in the graph $P(n,3)$
and let $y$ be a vertex or an edge in $P(n,3)$.
Denote by $\mathcal{P}(x,y)$ the set of shortest undeviating paths 
from $x$ to $y$.
Let $\mathcal{P}_-(x,y)$ be 
the set of clockwise paths in $\mathcal{P}(x,y)$,
and $\ell_-(x,y)$ the length of any path in $\mathcal{P}_-(x,y)$,
called the \emph{clockwise distance} from $x$ to $y$.
Let $\mathcal{P}_+(x,y)$ be 
the set of counterclockwise paths in $\mathcal{P}(x,y)$,
and $\ell_+(x,y)$ the length of any path in $\mathcal{P}_+(x,y)$,
called the \emph{counterclockwise distance} from $x$ to $y$.

\begin{lemma}\label{lem:dist:min}
Let $3\le i\le n-3$. Then we have the following in the graph $P(n,3)$.
\begin{enumerate}
\item\label[itm]{itm:l:uu}
$\ell_-(u_0,u_i)=q_i+r_i+2$ and $\ell_+(u_0,u_i)=q_n-q_i+r+2$.
\item\label[itm]{itm:l:uv}
$\ell_-(u_0,v_i)=q_i+r_i+1$
and $\ell_+(u_0,v_i)=q_n-q_i+r+1$.
Furthermore, $\abs{P_-(u_0,v_i)}=\abs{P_+(u_0,v_i)}=1$.
\item\label[itm]{itm:l:vv}
The clockwise and counterclockwise distances from $v_0$ to $v_i$ are respectively
\[
\ell_-(v_0,v_i)=\begin{cases}
q_i,&\text{if $r_i=0$}\\
q_i+r_i+2,&\text{if $r_i\ne 0$}
\end{cases}
\quad\text{and}\quad
\ell_+(v_0,v_i)=\begin{cases}
q_{n-i},&\text{if $r_n=r_i$}\\
q_{n-i}+r_{n-i}+2,&\text{if $r_n\ne r_i$}.
\end{cases}
\]
\end{enumerate}
%\begin{align*}
%d(v_0,v_i)&=\begin{cases}
%\min(q_i,\,q_n-q_i),&\text{if $r_i=0$ and $r_n=0$},\\
%\min(q_i,\,q_n-q_i+r_n+2),&\text{if $r_i=0$ and $r_n\ne 0$},\\
%\min(q_i+r_i+2,\,q_n-q_i),&\text{if $r_i\ne 0$ and $r_n=r_i$},\\
%\min(q_i+r_i+2,\,q_n-q_i+r+2),
%&\text{if $r_i\ne 0$ and $r_n\ne r_i$}.
%\end{cases}
%\end{align*}
\end{lemma}
\begin{proof}
It is elementary to compute that 
\begin{align}
(q_{n-i},\,r_{n-i})
&=\begin{cases}
(q_n-q_i,\,r_n-r_i),&\text{if $r_n\ge r_i$}\\
(q_n-q_i-1,\,1),&\text{if $(r_n,\,r_i)=(0,2)$}\\
(q_n-q_i-1,\,2),&\text{otherwise}
\end{cases}\label{eq:qr:n-i}
%\\
%&=\begin{cases}
%(q_n-q_i,\,r),&\text{if $r_n\ge r_i$}\\
%(q_n-q_i-1,\,r+1),&\text{if $r_n<r_i$}.
%\end{cases}
\end{align}

\noindent\cref{itm:l:uu}
Consider $p\in P_-(u_0,u_i)$.
By \cref{thm:undeviating},
the path $p$ has either 0 or 2 spokes.
In the former case, $\ell_p=i$;
in the latter case, $\ell_p=q_i+r_i+2$
for $p$ has $q_i$ steps on an inner cycle between the two spokes
and $r_i$ steps on the outer cycle.
Since $i\ge 3$, we can deduce that 
\[
\ell_-(u_0,u_i)=\min(i,\,q_i+r_i+2)=q_i+r_i+2.
\]
By symmetry and by \cref{eq:qr:n-i}, $\ell_+(u_0,u_i)=q_{n-i}+r_{n-i}+2=q_n-q_i+r+2$.

\noindent\cref{itm:l:uv}
Consider $p\in P_-(u_0,v_i)$.
By \cref{thm:undeviating},
the path $p$ has exactly one spoke.
Then 
\[
p=u_0u_1\dotsm u_{r_i}v_{r_i}v_{r_i+3}\dotsm v_i
\]
is unique,
with length $\ell_p=q_i+r_i+1$.
Indeed, the unique spoke in $p$ must be immediately after 
the first $r_i$ steps on the outer cycle.
By symmetry, $\abs{P_+(u_0,v_i)}=1$. By \cref{eq:qr:n-i}, 
\[
\ell_+(u_0,v_i)
=q_{n-i}+r_{n-i}+1
=q_n-q_i+r+1.
\]

\noindent\cref{itm:l:vv}
Consider $p\in P_-(v_0,v_i)$.
By \cref{thm:undeviating},
the path $p$ has either 0 or 2 spokes.
If $r_i=0$, then $p=v_0v_3\dotsm v_i$ is unique and $\ell_p=q_i$.
Otherwise $r_n\in\{1,2\}$.
If $p$ has no spokes, then $\ell_p=i$;
if $p$ has exactly two spokes, then $\ell_p=q_i+r_i+2$
for $p$ has $q_i$ steps on an inner cycle between the two spokes
and $r_i$ steps on the outer cycle. This proves the first desired formula.
It is direct to obtain the other formula by symmetry.
\end{proof}

Note that $\ell_\pm(u_0,u_i)=\ell_\pm(u_0,v_i)+1$,
and the path obtained by adding the arc $v_iu_i$ to the unique path 
in $P_\pm(u_0,v_i)$ belongs to $P_\pm(u_0,u_i)$.
For $n\in\mathbb{Z}$,
define $h\in\mathbb{Z}$ and $r_n'\in\{0,1,\dots,5\}$ by $n=6h+r_n'$.
Define $M_n=\max\{j<\floor{n/2}\colon r_j=r_n\}$.

\begin{theorem}\label{thm:dist:vv:i<n/2}
Let $0\le i\le \floor{n/2}$. 
Then we have the following in the graph $P(n,3)$.
\begin{align*}
d(u_0,u_i)&=\begin{cases}
q_i+r_i,&\text{if $i\le 2$}\\
q_i+r_i+1,&\text{if $(r_n',\,i)=(5,\,\floor{n/2})$}\\
q_i+r_i+2,&\text{otherwise},
\end{cases}\\
d(u_0,v_i)&=\begin{cases}
q_i+r_i,&\text{if $(r_n',\,i)=(5,\,\floor{n/2})$}\\
q_i+r_i+1,&\text{otherwise},
\end{cases}\\
d(v_0,v_i)&=\begin{cases}
q_i+r_i-1,&\text{if $(r_n',i)=(5,\floor{n/2})$}\\
q_i+r_i,&\text{if $r_n'\in\{2,4\}$ and $i=M_n$,
or $r_i=0$}\\
q_i+r_i+1,&\text{if $r_n'\in\{1,5\}$ and $i=M_n$}\\
q_i+r_i+2,&\text{otherwise}.
\end{cases}
\end{align*}
\end{theorem}
\begin{proof}
Let $0\le i\le \floor{n/2}$.
We show them individually.

Consider $d(u_0,u_i)$.
If $i\le 2$, it is easy to check that $d=q_i+r_i$.
Let $3\le i\le \floor{n/2}$.
Suppose that $\ell_+(u_0,u_i)<\ell_-(u_0,u_i)$.
By \cref{lem:dist:min}, $q_n-q_i+r+2<q_i+r_i+2$.
It is elementary to show that $(r_n',\,i)=(5,\,\floor{n/2})$.
In this case, $(q_n,r_n,q_i,r_i,r)=(2h+1,\,2,h,2,0)$ and
\[
d(u_0,u_i)
=q_n-q_i+r+2
=2h+1-h+0+2
=q_i+r_i+1.
\]
Since $r_{n-i}=0$, the set $P$ consists of a unique path, 
which has no edge in the outer cycle.

Consider $d(u_0,v_i)$.
Suppose that $\ell_+(u_0,v_i)<\ell_-(u_0,v_i)$.
By \cref{lem:dist:min} and the previous case, 
we find
$(r_n',i)=(5,\floor{n/2})$
and $d(u_0,u_i)=q_i+r_i$.

Consider $d=d(v_0,v_i)$. Suppose that $\ell_+(v_0,v_i)<\ell_-(v_0,v_i)$.
By \cref{lem:dist:min},
it is elementary to show that $r_n=r_i\ne 0$, $d=q_{n-i}$,
and
\begin{equation}\label[ineq]{pf:ineq:vv}
q_n-q_i\le q_i+r_i+1.
\end{equation}
We proceed according to the value of $r_n'$.
\begin{itemize}
\item
If $n=6h+1$, then $i\le 3h-2$ since $r_i=1$.
By \cref{pf:ineq:vv}, we find $i=3h-2=M_n$ and $d=h+1=q_i+r_i+1$.
\item
If $n=6h+2$, then $i\le 3h-1$ since $r_i=2$.
By \cref{pf:ineq:vv}, we find $i=M_n$ and $d=q_i+r_i$.
\item
If $n=6h+4$, then $i\le 3h+1$ since $r_i=1$.
By \cref{pf:ineq:vv}, we find $i=M_n$ and $d=q_i+r_i$.
\item
If $n=6h+5$, then $i\le 3h+2$ since $r_i=2$.
By \cref{pf:ineq:vv}, we find that either $i=3h+2=\floor{n/2}$ 
or $i=3h-1=M_n$. In the former case, $d=h+1=q_i+r_i-1$;
in the latter case, $d=h+2=q_i+r_i+1$.
\end{itemize}
This completes the proof.
\end{proof}

Now we start Step 3.

\begin{proposition}\label{prop:qr<qr}
Let $i,j\in\mathbb{Z}$ and $j>i$. 
We have the following equivalence:
\[
q_{j}+r_{j}<q_i+r_i
\iff
j=i+1\quad\text{and}\quad r_i=2.
\]
In this case, $q_{i+1}+r_{i+1}+1=q_i+r_i$.
\end{proposition}
\begin{proof}
It is elementary and we omit the proof.
\end{proof}

\begin{lemma}\label{lem:eu:i<n/2}
Let $0\le i\le \ceil{n/2}-1$.
Then the following equivalences hold:
\begin{align*}
d(u_0,u_{i+1})<d(u_0,u_i)
&\iff
i\in\{5,\,8,\,11,\,\dots,\,3\floor{n/6}-1\},\\
d(v_0,u_{i+1})<d(v_0,u_i)
&\iff
i\in\{2,\,5,\,8,\,\dots,\,3\floor{n/6}-1\}.
\end{align*}
\end{lemma}
\begin{proof}
We show the two equivalences individually.

Let $d_1=d(u_0,u_i)$ and $d_2=d(u_0,u_{i+1})$.
Suppose that $d_2<d_1$. We proceed by contradiction.

\noindent\emph{Case 1. $i\le 2$.}
Then $(d_1,d_2)=(i,\,i+1)$, contradicting the premise $d_2<d_1$.

\noindent\emph{Case 2. $r_n'=5$ and $i=\floor{n/2}$.}
Then $n$ is odd and $d_2=d(u_0,u_{n-i-1})=d_1$.

\noindent\emph{Case 3. $(r_n',i)\ne (5,\,\floor{n/2}-1)$}.
Then $(n,i)=(6h+5,\,3h+1)$.
By \cref{thm:dist:vv:i<n/2}, $d_1=h+3=d_2$.

\noindent\emph{Case 4. None of the above.}
By \cref{thm:dist:vv:i<n/2}, the premise $d_2<d_1$ reduces to the inequaltiy in \cref{prop:qr<qr}, which is equivalent to $r_i=2$; in this case, $d_2=d_1-1$
by \cref{prop:qr<qr}.

Rearranging the above results, we obtain the first desired equivalence.

Now, let $d_1=d(v_0,u_i)$
and $d_2=d(v_0,u_{i+1})$. 
Suppose that $d_2<d_1$. We treat 3 cases.

\noindent\emph{Case 1. $i=\ceil{n/2}-1$.}
If $n$ is odd, then $d_1=d_2$ by symmetry, a contradiction.
Suppose that $n$ is even.
By \cref{thm:dist:vv:i<n/2}, the assumption $d_2<d_1$ reduces to 
the inquality in \cref{prop:qr<qr},
which is equivalent to $r_i=2$.
Since $i=n/2-1$, we find $r_n'=0$.
%In this case, $d_2=q_{i+1}+r_{i+1}+1=q_i+r_i$.

\noindent\emph{Case 2. $(r_n',i)=(5,\,\ceil{n/2}-2)$.}
Then $i=3h+1$ and $d_1=h+2=d_2$, a contradiction.

\noindent\emph{Case 3. None of the above.}
Since $i+1\le \floor{n/2}$, by \cref{thm:dist:vv:i<n/2,prop:qr<qr},
the assumption $d_2<d_1$ reduces to $r_i=2$.
%As a result, $d_2=q_{i+1}+r_{i+1}+1=q_i+r_i$.

Rearranging the above results, we obtain the second desired equivalence.
\end{proof}

\begin{lemma}\label{lem:es:i<n/2}
Let $0\le i\le \floor{n/2}$. 
Then the following equavilences hold:
\begin{align*}
d(u_0,u_i)<d(u_0,v_i)
&\iff
i\in\{0,1,2\}, \\
d(v_0,v_i)<d(v_0,u_i)
&\iff
r_i=0
\text{ or }
(r_n',i)\in\{(5,\floor{n/2}),\,(2,M_n),\,(4,M_n)\}.
\end{align*}
\end{lemma}
\begin{proof}
Direct from \cref{thm:dist:vv:i<n/2}.
\end{proof}

\begin{lemma}\label{lem:u-ev:i<n/2}
Let $0\le i\le \ceil{n/2}-2$, $d_1=d(u_0,v_i)$ and $d_2=d(u_0,v_{i+3})$.
\begin{enumerate}
\item
If $(r_n',i)=(2,\,n/2-2)$,
then $d_1>d_2=q_i+r_i$.
\item
Otherwise, $d_2\ge d_1=q_i+r_i+1$.
\end{enumerate}
\end{lemma}
\begin{proof}
Let $0\le i\le \ceil{n/2}-2$ and $j=n-i-3$.
Then $d_2=d(u_0,v_j)$ by symmetry.
Suppose that $d_2<d_1$. We treat 3 cases.

\noindent\emph{Case 1. $i=\ceil{n/2}-2$.}
Then $j=\floor{n/2}-1$.
If $n$ is odd, then $j=i$ and $d_1=d_2$, a contradiction.
Otherwise $n$ is even, then $j=i+1$. By \cref{thm:dist:vv:i<n/2},
the assumption $d_2<d_1$ reduces to the inequaltiy in \cref{prop:qr<qr}, which is equivalent to $r_i=2$.
It follows that $r_n'=2$
and $d_2=h+1=q_i+r_i$.

\noindent\emph{Case 2. $i=\ceil{n/2}-3$ and $n$ is odd.}
Then $j=\floor{n/2}=i+2$. 
If $r_n'=5$, then $(r_i,r_j)=(0,2)$. 
On the other hand, by \cref{thm:dist:vv:i<n/2},
the assumption $d_2<d_1$ reduces to $q_j+r_j<q_i+r_i+1$, i.e., $q_j+1< q_i$,
which is impossible since $j=i+2$.
Otherwise $r_n'\in\{1,3\}$. 
Then $(r_i,r_j)=\{(1,0),\,(2,1)\}$ and $r_j=r_i-1$. 
By \cref{thm:dist:vv:i<n/2}, 
the assumption $d_2<d_1$ reduces to $q_j+r_j+1<q_i+r_i+1$, i.e.,
$q_j\le q_i$. Since $j=i+2$, 
we find $q_j=q_i$. It follows that $r_j\ge r_i$, a contradiction.

\noindent\emph{Case 3. None of the above.}
Then $i+3\le \floor{n/2}$. 
By \cref{thm:dist:vv:i<n/2},
\begin{equation}\label{pf:d2-d1}
d_2-d_1=q_{i+3}+r_{i+3}-q_i-r_i=1,
\end{equation}
contradicting the assumption $d_2<d_1$.

Therefore, by \cref{thm:dist:vv:i<n/2},
as if $(r_n',i)\ne (2,\,n/2-2)$,
$d_2\ge d_1=q_i+r_i+1$.
\end{proof}

\begin{lemma}\label{lem:v-ev:i<n/2}
Let $0\le i\le \ceil{n/2}-2$, $d_1=d(v_0,v_i)$
and $d_2=d(v_0,v_{i+3})$. Suppose that $d_2<d_1$.
Then one of the following is true.
\begin{enumerate}
\item
$r_n'=2$ and $(i,d_2)=(M_n,\,q_r+r_i-1)$.
\item
$r_n'\in\{1,5\}$ and $(i,d_2)=(M_n,\,q_r+r_i)$.
\item
$r_n'\in\{2,4\}$ and $(i,d_2)=(M_n-3,\,q_r+r_i+1)$.
\end{enumerate}
\end{lemma}
\begin{proof}
Suppose that $r_n=0$.
When $i+3\le n/2$, we can deduce \cref{pf:d2-d1}
by \cref{thm:dist:vv:i<n/2},
contradicting the premise $d_2<d_1$.
Consider the other case $n/2-3<i\le \ceil{n/2}-2$.
If $n=6h$, then
$i=3h-2$ and $d_1=h+2<h+3=d_2$, the same contradiction.
Otherwise $n=6h+3$.
Then $i=3h-1$ and $d_1=d_2=h+3$, the same contradiction.

Below we can suppose that $r_n\ne0$.
When $i+3<M_n$, we obtain the same contradiction \cref{pf:d2-d1}  to the premise $d_2<d_1$. It remains to compute
the pair $(d_1,d_2)$ for $M_n-3\le i\le \ceil{n/2}-2$.
We observe that the upper bound can be further improved to $\floor{n/2}-2$ when $n$ is odd since $d_1=d_2$ by symmetry when $i=(n+1)/2-2$.
By \cref{thm:dist:vv:i<n/2}, we find \cref{tab:d1d2:vv},
from which we see that

\begin{itemize}
\item
if $r_n'\in\{1,5\}$, then $(i,d_2)=(M_n,\,q_i+r_i)$;
\item
if $r_n'=2$, then $(i,d_2)=(M_n,\,q_i+r_i-1)$ or $(i,d_2)=(M_n-3,\,q_i+r_i+1)$; and
\item
if $r_n'=4$, then $(i,d_2)=(M_n-3,\,q_i+r_i+1)$.
\end{itemize}
\begin{table}[h]
\centering
\caption{The distances $d_1$ and $d_2$ when $M_n-3\le i\le \floor{n/2}-2$.}
\label{tab:d1d2:vv}
\begin{tabular}{@{} CCCCCCCCC @{}}\toprule
&&\multicolumn{6}{C}{i}\\
\cmidrule(r){3-8}
n & & 3h & 3h-1 & 3h-2 & 3h-3 & 3h-4 & 3h-5 \\
\midrule
6h+1 
& d_1
&  &  & h+1 & h-1 & h+2 & h+1 \\ 
& d_2 
&  &  & h & h & h+3 & h+1 \\
6h+2
& d_1
&  & h+1 & h+2 & h-1 & h+2 &  \\ 
& d_2 
&  & h & h+3 & h &  h+1 &  \\
6h+4
& d_1
& h+3 & h+3 & h+2 &  &  &  \\ 
& d_2 
& h+3 & h+4 &  h+1 &  &  &  \\ 
6h+5
& d_1 & h & h+2 & h+2 & h-1 & h+2 & \\ 
& d_2 & h+1 & h+1 & h+3 & h &  h+2 & \\
\bottomrule
\end{tabular}
\end{table}
This completes the proof.
\end{proof}

Now, combining \cref{lem:eu:i<n/2,lem:es:i<n/2,lem:u-ev:i<n/2,lem:v-ev:i<n/2},
it is elementary to show \cref{thm:dist:ve} for when $i<n/2$,
and easy to derive the formulas for $i\ge n/2$ by symmetry. 
We leave the proof details to the readers.

\section{Proof of \cref{prop:edim>=4}}\label[sec]{sec:edim>=4}
This section is devoted to proving that the edge dimension of 
the graph $P(n,3)$ is at least 4.
Looking for a pair of edges that have the same distance to a given vertex,
we focus on pairs of edges of the form $(e_{i-1}^u,\,e_i^u)$.

\begin{lemma}\label{lem:A}
Let $n\ge 100$ and $A_0=\brk[c]{i\in\mathbb{Z}_n\colon d(u_0,\,e_{i-1}^u)=d(u_0,\,e_i^u)}$. Then
\[
A_0=\begin{cases}
\{\pm i\in\mathbb{Z}\colon 5\le i<n/2\text{ and }r_i\ne 1\}\cup\{0,\,n/2\},
&\text{if $n$ is even};\\
\{\pm i\in\mathbb{Z}\colon 5\le i<n/2\text{ and }r_i\ne 1\}\cup\{0\},
&\text{if $n$ is odd}.
\end{cases}
\]
\end{lemma}
\begin{proof}
By symmetry, $n/2\in A_0$ when $n$ is even.
By \cref{thm:dist:ve},
it is direct to verify that $0\in A_0$
and $\{\pm1,\,\pm2,\,\pm3,\,\pm4\}\cap A_0=\emptyset$.
Consider $5\le i<n/2$. By \cref{thm:dist:ve}, 
\begin{align*}
d(u_0,\,e_{i-1}^u)
&=2+\min\brk1{\ceil{(i-1)/3},\,\ceil{(n-i)/3}}=2+\ceil{(i-1)/3},\\
d(u_0,\,e_i^u)
&=2+\min\brk1{\ceil{i/3},\,\ceil{(n-i-1)/3}}=2+\ceil{i/3}.
\end{align*}
Hence we derive the equivalences
$
i\in A_0
\iff
\ceil{(i-1)/3}=\ceil{i/3}
\iff
r_i\ne1$.
By symmetry, we know that $-i\in A_0\iff r_i\ne1$.
This completes the proof.
\end{proof}

\begin{lemma}\label{lem:D}
Let $n\ge 100$ and $B_0=\brk[c]{i\in\mathbb{Z}_n\colon d(v_0,\,e_{i-1}^u)=d(v_0,\,e_i^u)}$. Then
\[
B_0=\begin{cases}
\{\pm i\in\mathbb{Z}\colon 0\le i<n/2\text{ and }r_i\ne 1\}\cup\{n/2\},
&\text{if $n$ is even};\\
\{\pm i\in\mathbb{Z}\colon 0\le i<n/2\text{ and }r_i\ne 1\},
&\text{if $n$ is odd}.
\end{cases}
\]
As a consequence, $A_0\subseteq B_0$.
\end{lemma}
\begin{proof}
It is clear that $0\in B_0$.
By symmetry, $n/2\in B_0$ when $n$ is even.
Consider $1\le i<n/2$.
Let $d_1=d(v_0,\,e_{i-1}^{u})$ and $d_2=d(v_0,\,e_i^u)$.
By \cref{thm:dist:ve},
it is direct to check that for $0\le i\le\ceil{n/2}-1$,
\[
d(v_0,u_{i+1})<d(v_0,u_i)
\iff
i\in\{2,\,5,\,8,\,\dots,\,3\floor{n/6}-1\}.
\]
Therefore, if $r_i=0$, then $d_1=d(v_0,u_i)=d_2$ and $i\in B_0$.

Suppose that $r_i=1$. Then $(q_{i-1},\,r_{i-1})=(q_i,0)$.
Since $3q_n+r_n=n>2i=6q_i+2$, we find $q_n>2q_i$,
which implies that 
\begin{equation}\label[ineq]{pf:ineq:qn-qi:D0}
q_n-q_i\ge q_i+1.
\end{equation}
By \cref{thm:dist:ve},
\begin{align*}
d_1
&=\min(q_{i-1}+\floor{r_{i-1}/2}+1,\,q_n-q_{i-1}+\floor{r_n/2}-\floor{r_{i-1}/2}+1)\\
&=\min(q_i+1,\,q_n-q_i+\floor{r_n/2}+1)=q_i+1,\quad\text{and}\quad\\
d_2&=\min(q_i+2,\,q_n-q_i+1)=q_i+2.
\end{align*}
Therefore, $d_2\ne d_1$. This proves that $i\not\in B_0$.
 
Suppose that $r_i=2$.
Then $(q_{i-1},\,r_{i-1})=(q_i,1)$.
Along the same line we can derive \cref{pf:ineq:qn-qi:D0}.
When $r_n=0$, one may enhence it to $q_n-q_i\ge q_i+2$ since $3q_n=n>2i=6q_i+4$.
By \cref{thm:dist:ve},
\begin{align*}
d_1
&=\min(q_{i-1}+2,\,q_n-q_{i-1}+1)
=\min(q_i+2,\,q_n-q_i+1)
=q_i+2,\\
d_2
&=\begin{cases}
\min(q_i+\floor{r_i/2}+1,\,q_n-q_i+\floor{r_n/2}-\floor{r_i/2}+1),
&\text{if $r_n=0$}\\
\min(q_i+2,\,q_n-q_i+1),
&\text{if $r_n\ne 0$}
\end{cases}\\
&=\begin{cases}
\min(q_i+2,\,q_n-q_i),
&\text{if $r_n=0$}\\
\min(q_i+2,\,q_n-q_i+1),
&\text{if $r_n\ne 0$}
\end{cases}\\
&=q_i+2.
\end{align*}
By symmetry, we obtain the desired formula for $B_0$.
As a consequence, $A_0\subseteq B_0$ by \cref{lem:A}.
\end{proof}

For $t\in\mathbb{Z}_n$, define $A_t=A_0+t$ and $B_t=B_0+t$. 
From the definitions, we see that
\begin{align*}
d(u_t,\,e_{i-1}^u)&=d(u_t,\,e_i^u)\quad\text{for $i\in A_t$},\quad\text{and}\\
d(v_t,\,e_{j-1}^u)&=d(v_t,\,e_j^u)\quad\text{for $j\in B_t$}.
\end{align*}
By \cref{lem:D}, we obtain 
\begin{equation}\label{At:Bt}
A_t\subseteq B_t\qquad \text{for any $t\in\mathbb{Z}_n$}.
\end{equation}

\begin{lemma}\label{lem:3:distinct}
Let $n\ge 100$.
 If the graph $P(n,3)$ has an edge resolving triad $R$,
then the vertices in $R$ have distinct subscripts.
\end{lemma}
\begin{proof}
Suppose that $P(n,3)$ has an edge resolving set $R=\{\alpha_x,\beta_y,\gamma_z\}$,
where $\alpha,\beta,\gamma\in\{u,v\}$ and $x,y,z\in\mathbb{Z}_n$.
Let $T=\{x,y,z\}$. Then $\abs{T}\in\{2,3\}$.

Assume that $\abs{T}=2$. By symmetry, we can suppose that $R=\{u_0,v_0,\gamma_t\}$, where $1\le t\le\floor{n/2}$. 
By \cref{lem:A,lem:D},
it suffices to show that
\begin{itemize}
\item
$A_0\cap B_0\cap A_t\ne \emptyset$ if $\gamma=u$, and that
\item
$A_0\cap B_0\cap B_t\ne \emptyset$ if $\gamma=v$.
\end{itemize} 
In view of \eqref{At:Bt}, it suffices to show $A_0\cap A_t\ne \emptyset$.
In fact, by \cref{lem:A}, it is easy to read out an element in the intersection; see \cref{tab:00:eta:t}.
\begin{table}[htbp]
\centering
\caption{An element in the intersection $A_0\cap A_t$, when $n\ge 100$.}
\label{tab:00:eta:t}
\begin{tabular}{@{} CCCC @{}}\toprule
t & r_t=0 & r_t=1 & r_t=2  \\ 
\midrule
\le \floor{n/2}
& t+6
& t+5 
& t+6 \\[3pt]
>\floor{n/2}
& t-6
& t-5
& t-5 \\
\bottomrule
\end{tabular}
\end{table}
\end{proof}

\Cref{lem:3:distinct} leads us to consider vertex triads with distinct subscripts.
We say that two triples $T=(x,y,z)\in\mathbb{Z}_n^3$ and $T'\in\mathbb{Z}_n^3$
are \emph{equivalent} if $T'=(x+a,\,y+a,\,z+a)$ for some $a\in\mathbb{Z}_n$, 
or $T'=(2t-x,\,2t-y,\,2t-z)$ for some $t\in\mathbb{Z}_n$, i.e., 
if $T'$ is a translation or reflection of~$T$.
It is clear that a vertex triad $\{\alpha_x,\beta_y,\gamma_z\}$,
where $\alpha,\beta,\gamma\in\{u,v\}$ and where $x,y,z\in\mathbb{Z}_n$ are distinct,
is an edge resolving set
if and only if so is every vertex set $\{\alpha_{x'},\beta_{y'},\gamma_{z'}\}$,
where $(x',y',z')$ is equivalent to $(x,y,z)$.
Considering the representatives of the equivalence classes,
we produce \cref{prop:Sn}.

\begin{proposition}\label{prop:Sn}
For any 3 distinct elements $x,y,z\in\mathbb{Z}_n$,
there exist $(a,b)\in S_n$ such that the triples 
$(x,y,z)$ and $(0,a,b)$ are equivalent.
\end{proposition}
\begin{proof}
For any $x,y\in\mathbb{Z}_n$, define 
$d(x,y)=\min\brk1{\abs{x'-y'},\,n-\abs{x'-y'}}$,
where $x'$ and $y'$ are the least nonnegative residue of $x$ and $y$ modulo $n$,
respectively.
By symmetry, we can suppose that 
\[
0=x<y<z<n
\quad\text{and}\quad
d(x,y)\le \min\brk1{d(y,z),\,d(z,x)}.
\]
It follows that $d(0,y)=y$ and $2y\le z\le n-y$. Thus $y\le \floor{n/3}$.
If $z\le \floor{(n+y)/2}$, then we can take $(a,b)=(y,z)$.
It remains to consider when 
\begin{equation}\label[ineq]{pf:floor1}
z\ge\floor{(n+y)/2}+1.
\end{equation}
In this case,
we translate the triple $(0,y,z)$ to $T'=(-y,\,0,\,z-y)$ by substituting $y$ from each coordinate,
and reflect~$T'$ to $T''=(y,\,0,\,y-z)$ about the central axis going through $0$.
Let $a=y$ and $b=y-z+n$. 
Then $T''=(0,a,b)$ forms an edge resolving set. It is routine 
to verify that $1\le a\le \floor{n/3}$ and $2a\le b$.
If $b\le \floor{(n+a)/2}$, then we are done.
Otherwise $b\ge \floor{(n+a)/2}+1$, i.e.,
\begin{equation}\label[ineq]{pf:floor2}
y-z+n\ge\floor{(n+y)/2}+1.
\end{equation}
Adding \cref{pf:floor1,pf:floor2} yields $t\ge 2\floor{t/2}+2$ for $t=y+n$,
which is absurd.
\end{proof}

By \cref{prop:Sn},
it suffices to show that for each $(a,b)\in S_n$,
no triad in the set $T(a,b)$
is an edge resolving set. 
We produce \cref{lem:AAA},
which works for the $\abs{T(a,b)}=8$ triads altogether. 

\begin{lemma}\label{lem:AAA}
Let $n\ge 100$ and $(a,b)\in S_n$.
If $A_0\cap A_a\cap A_b\ne \emptyset$,
the no triad in the set $T(a,b)$ is an edge resolving set.
\end{lemma}
\begin{proof}
Suppose that $A_0\cap A_a\cap A_b\ne \emptyset$.
By \cref{lem:A}, the vertex triad $\{u_0,u_a,u_b\}$ is not a resolving set.
By \eqref{At:Bt}, any other triad in $T(a,b)$ is not a resolving set.
\end{proof}

Inspired by \cref{lem:AAA},
we are going to determine, as more as possible, the pairs $(a,b)\in S_n$ 
such that $A_0\cap A_a\cap A_b=\emptyset$.
To do this,
we define a set $W_n\subseteq S_n$ by \cref{tab:def:W}.
\begin{table}[htbp]
\centering
\caption{Definition of the set $W_n$ for $n\ge 100$, where $r_n'$ is the residue of $n$ modulo 6.}
\label{tab:def:W}
\begin{tabular}{@{} CLC @{}}
\toprule
r_n' & W_n & \abs{W_n} \\ 
\midrule
0
& \emptyset
& 0 \\[3pt]
1 
& \brk[c]{(1,\,\floor{n/2}-2),\,(2,\,\floor{n/2}),\,(5,\,\floor{n/2}+3)}
& 3 \\[3pt]
2 
& \emptyset
& 0 \\[3pt]
3 
& \brk[c]{(1,2),\,(2,4)}
& 2 \\[3pt]
4
& \brk[c]1{(1,2),\,
\brk{1,\frac{n}{2}-4},\,
\brk{1,\frac{n}{2}-1},\,
\brk{2,\frac{n}{2}-2},\,
\brk{2,\frac{n}{2}+1},\,
\brk{4,\frac{n}{2}-1},\,
\brk{5,\frac{n}{2}+1},\,
\brk{8,\frac{n}{2}+4}}
& 8\\[3pt]
5 
& \brk[c]{(1,2),\,(1,5),\,(1,8),\,(2,4),\,(2,7),\,(2,10),\,(4,8),\,(4,11),\,(7,14)}
& 9\\
\bottomrule
\end{tabular}
\end{table}

Now, to prove \cref{prop:edim>=4}, 
it suffices to establish \cref{prop:S-W,lem:W}.
\begin{proposition}\label{prop:S-W}
Let $n\ge 100$ and $(a,b)\in S_n$.
If $A_0\cap A_a\cap A_b=\emptyset$, then $(a,b)\in W_n$.
\end{proposition}
\begin{lemma}\label{lem:W}
Let $n\ge 100$ and $(a,b)\in W_n$. For any $R\in T(a,b)$,
there exists a pair of edges with the same edge metric representation
with respect to $R$.
\end{lemma}
We will show \cref{prop:S-W,lem:W} in \cref{sec:S-W,sec:W}, respectively.

\subsection{Proof of \cref{prop:S-W}}\label[sec]{sec:S-W}
Throughout this section,
we suppose that $n\ge 100$ and $(a,b)\in S_n$.
Let $J=A_0\cap A_a\cap A_b$.
We will establish \cref{prop:S-W} for $r_n'\in\{2,5,0\}$
by combinatorial arguments, and for $r_n'\in\{3,1,4\}$ 
by explicitly constructing an element in $J$ as if $(a,b)\not\in W_n$.
For any integers $m$ and~$M$, we denote $[m,M]=\{m,m+1,\dots,M\}$.
When $r_m=r_M$, we define 
\[
g(m,M)=\{m,\,m+3,\,m+6,\,\dots,\,M\}.
\]

\begin{lemma}\label{lem:fence}
Let $m<M$ be integers such that $r_m=r_M$.
Suppose that $g(m,M)\subseteq A_x\cap A_y$ and $z\in[m-6,\,M+9]$. 
If $A_x\cap A_y\cap A_z=\emptyset$,
then 
\[
z\in\{m-4,\,m-1,\,m+2,\,M-2,\,M+1,\,M+4\}.
\]
\end{lemma}
\begin{proof}
Let $J=A_x\cap A_y\cap A_z$. Then we have the following.
\begin{itemize}
\item
If $r_{z-m}=0$, then $t\in J$, where
\[
t=\begin{cases}
z,&\text{if $z\in g(m,M)$};\\
z+6,&\text{if $z\in\{m-3,\,m-6\}$};\\
z-6,&\text{if $z\in\{M+3,\,M+6\}$}.
\end{cases}
\]
\item
If $r_{z-m}=1$ and $m-5\le z\le M-5$, then $z+5\in J$.
\item
If $r_{z-m}=2$ and $m+5\le b\le M+5$, then $z-5\in J$.
\end{itemize}
Since $J=\emptyset$ by premise, we obtain the desired result.
\end{proof}
To use \cref{lem:fence} smoothly, we point out that 
for any $(a,b)\in S_n$,
\begin{equation}\label[ineq]{ub:b}
b\le \frac{n+a}{2}\le \frac{2n}{3}<n-8.
\end{equation}

\subsubsection{\cref{prop:S-W} is true if $r_n'=2$}
\label[sec]{ssec:6h+2}
Suppose that $r_n'=2$ and $J=\emptyset$.
By \cref{lem:A},
\begin{equation}\label{intersection:A0a:rn=2}
A_0\cap A_a\supseteq\begin{cases}
g(a+5,\,n-6)\cup g(a+6,\,n-5),&\text{if $r_a=0$};\\
g(a+5,\,n-5),&\text{if $r_a=1$};\\
g(a+6,\,n-6),&\text{if $r_a=2$}.
\end{cases}
\end{equation}
By \cref{lem:fence,ub:b}, 
we deduce that $b\le a+8$.

For any $t,j\in\mathbb{Z}_n$, define
\[
F_{tj}=t+n/2+\{0,\,\pm 1,\,\pm 2,\,\dots,\,\pm j\}.
\]
By \cref{lem:A}, we obtain $F_{t2}\subseteq A_t$.
If $x,y\not\in F_{z4}$, then
$\abs{A_x\cap F_{z2}}\ge 3$
and 
$\abs{A_y\cap F_{z2}}\ge 3$.
Since $\abs{F_{z2}}=5$, we deduce that $A_x\cap A_y\cap F_{z2}\ne\emptyset$,
contradicting the premise $J=\emptyset$.
Henceforth, we can suppose that $F_{t4}\cap T\ne \emptyset$ for any $t\in T$.
Since $a\le n/3<n/2-4$, we deduce that $a\not\in F_{04}$.
Thus $F_{04}=\{b\}$. 
From the definition, we obtain
\begin{equation}\label[ineq]{pf:b}
n/2+a-4\le b\le \floor{(n+a)/2}.
\end{equation}
Hence $n/2+a-4\le a+8$,
contradicting the premise $n\ge 100$.
This proves \cref{prop:S-W} for $r_n'=2$.

\subsubsection{\cref{prop:S-W} is true if $r_n'=5$}
\label{ssec:6h+5}
Suppose that $r_n'=5$ and $J=\emptyset$. 
By \cref{lem:A}, we obtain the relation \eqref{intersection:A0a:rn=2}.
Now we use \cref{lem:fence} with the aids of \eqref{intersection:A0a:rn=2} and \cref{ub:b}.
If $r_a=0$, then $2a\le b\le a-2$, which is impossible.
If $a\in\{1,2,4\}$, it is direct to deduce that $(a,b)\in W_n$.
If $a\ge 5$ and $r_a=2$,
then $b\in\{a+5,\,a+8\}$. 
It follows that $a\in J$, contradicting the premise $J=\emptyset$.
If $a\ge 7$ and $r_a=1$, then $b=a+7$ and $(a,b)=(7,14)\in W_n$.
This completes the proof of \cref{prop:S-W} for $r_n'=5$.

\subsubsection{\cref{prop:S-W} is true if $r_n'=0$}\label[sec]{ssec:6h}
Suppose that $r_n'=0$ and $J=\emptyset$.
Consider the general case 
$T=(x,y,z)\subset \mathbb{Z}_n$ and $J=A_x\cap A_y\cap A_z$.

First, from \cref{lem:A}, we see that $F_{t1}\subseteq A_t$.
If $x,y\not\in F_{z4}$, then
\[
\abs{A_x\cap F_{z1}}\ge 2
\quad\text{and}\quad
\abs{A_y\cap F_{z1}}\ge 2.
\]
Since $\abs{F_{z1}}=3$, we deduce that $A_x\cap A_y\cap F_{z1}\ne\emptyset$,
contradicting the premise $J=\emptyset$.
Henceforth, we can suppose that $F_{t4}\cap T\ne \emptyset$ for any $t\in T$.
As \cref{ssec:6h+2}, we can deriv \cref{pf:b}.

Second, we claim that the elements in $T$ form a complete residue class modulo 3.
Let 
\[
Y_j=\{i\in[\,0,\,n-1]\colon r_i=j\}
\quad\text{and}\quad
A_t^c=\mathbb{Z}_n\backslash A_t.
\]
Suppose that $n=6h$ for some $h\in\mathbb{Z}$.
By \cref{lem:A}, we find
\[
\abs1{A_t^c\cap Y_j}=\begin{cases}
2,&\text{if $j=r_t$};\\
h+1,&\text{otherwise}.
\end{cases}
\]
If $r_x=r_y$, then for $j=r_x$, 
\begin{align*}
\abs{A_x^c\cap Y_j}
+\abs{A_y^c\cap Y_j}
+\abs{A_z^c\cap Y_j}
\le 2+2+(h+1)<2h=\abs{Y_j}.
\end{align*}
Therefore, $J\cap Y_j\ne\emptyset$, contradicting the premise $J=\emptyset$. 
This proves the claim.

Now we are in a position to show that $i^*\in J$,
where $i^*=a+4+r_a$.
In fact, by \cref{pf:b} and the premise $n\ge 100$,
one may compute the distances
\[
d(i^*,\,0)=a+4+r_a,\quad
d(i^*,\,a)=4+r_a,\quad\text{and}\quad
d(i^*,\,b)=b-a-4-r_a.
\]
By \cref{pf:b} and the claim,
we can deduce that 
\begin{align*}
5\le a+4+r_a\not\equiv1\pmod{3}&\implies i^*\in A_0,\\
5\le 4+r_a\not\equiv1\pmod{3}&\implies i^*\in A_a,
\quad\text{and}\quad\\
5\le (n/2-4)-4-2\le b-a-4-r_a\equiv 2\pmod{3}
&\implies i^*\in A_b.
\end{align*}
This completes the prooof of \cref{prop:S-W} for $r_n'=0$.

\subsubsection{\cref{prop:S-W} is true if $r_n'\in\{3,1,4\}$}\label[sec]{ssec:6h+314}
For the remaining case that $r_n'\in \{3,1,4\}$,
we explicitly construct an element in $J$ as if $(a,b)\in S_n\backslash W_n$;
see \cref{tab:6h+3,tab:6h+4,tab:6h+1}.
It is routine to check them by using \cref{lem:A}.

\begin{table}[htbp]
\centering
\caption{An element in $A_0\cap A_a\cap A_b$, where $(a,b)\in S_n\backslash W_n$
and $n=6h+3\ge 100$.}
\label{tab:6h+3}
\begin{tabular}{@{} LLLL @{}}
\toprule
r_a & r_b=0 & r_b=1 & r_b=2 \\ 
\midrule
0 
& -6
& b+5
& \begin{cases}
3h-4,&\text{if $b\in \{3h-4,\,3h+2\}$}\\[3pt]
3h-1,&\text{otherwise}
\end{cases} \\ 
\addlinespace
1 %%%%%%%%%%%%%%%%%%%%%%%%%%%%%%%%%%%%%%%%%%%%%%%%%%%%%%%%%%%%
& \begin{cases}
-5,&\text{if $(a,b)=(1,3)$}\\
b,&\text{otherwise}
\end{cases}
& -5
& \begin{cases}
a+5,&\text{if $b-a>7$}\\[3pt]
3h+b-a,&\text{otherwise}
\end{cases}\\ 
\addlinespace
2  %%%%%%%%%%%%%%%%%%%%%%%%%%%%%%%%%%%%%%%%%%%%%%%%%%%%%%%%%%%%
& -6
& \begin{cases}
a,&\text{if $a>2$}\\[3pt]
3h+6,&\text{otherwise}
\end{cases} 
& \begin{cases}
a,&\text{if $a>2$}\\
11,&\text{if $a=2$, $b\in\{5,11\}$}\\
8,&\text{otherwise}
\end{cases} \\ 
\bottomrule
\end{tabular}
\end{table}

\begin{table}[htbp]
\centering
\caption{An element in $A_0\cap A_a\cap A_b$, where $(a,b)\in S_n\backslash W_n$
and $n=6h+1\ge 100$.}
\label{tab:6h+1}
\begin{tabular}{@{} LLLL @{}}
\toprule
r_a & r_b=0 & r_b=1 & r_b=2 \\ 
\midrule
0 
& -5
& \begin{cases}
-5,&\text{if $b\le 3h-5$}\\
3h-7,&\text{if $b\ge 3h-2$}
\end{cases}
& b \\ 
\addlinespace
1 %%%%%%%%%%%%%%%%%%%%%%%%%%%%%%%%%%%%%%%%%%%%%%%%%%%%%%%%%%%%
& -5
& \begin{cases}
-5,&\text{if $b\le 3h-5$}\\
3h+4,&\text{if $b=3h-2$}\\
b,&\text{otherwise}
\end{cases}
& \begin{cases}
3h+1,&\text{if $b\le 3h-4$}\\[3pt]
-5,&\text{otherwise}
\end{cases} \\ 
\addlinespace
2  %%%%%%%%%%%%%%%%%%%%%%%%%%%%%%%%%%%%%%%%%%%%%%%%%%%%%%%%%%%%
& b+5
& \begin{cases}
3h+4,&\text{if $b\le 3h-2$}\\
b,&\text{otherwise}
\end{cases}
& -6 \\ 
\bottomrule
\end{tabular}
\end{table}

\begin{table}[htbp]
\centering
\caption{An element in $A_0\cap A_a\cap A_b$, where $(a,b)\in S_n\backslash W_n$
and $n=6h+4\ge 100$.}
\label{tab:6h+4}
\begin{tabular}{@{} LLLL @{}}
\toprule
r_a & r_b=0 & r_b=1 & r_b=2 \\ 
\midrule
0 
& -5
& \begin{cases}
-5,&\text{if $b\le 3h-5$}\\[3pt]
3h-7,&\text{otherwise}
\end{cases}
& -6 \\ 
\addlinespace
1 %%%%%%%%%%%%%%%%%%%%%%%%%%%%%%%%%%%%%%%%%%%%%%%%%%%%%%%%%%%%
& -5
& \begin{cases}
-5,&\text{if $b\le 3h-5$}\\[3pt]
b,&\text{if $b\ge 3h+4$}\\[3pt]
b+6,&\text{otherwise}
\end{cases}
& \begin{cases}
b+3h,&\text{if $b\le 3h-4$}\\[3pt]
-5,&\text{otherwise}
\end{cases} \\ 
\addlinespace
2  %%%%%%%%%%%%%%%%%%%%%%%%%%%%%%%%%%%%%%%%%%%%%%%%%%%%%%%%%%%%
& b+5
& \begin{cases}
3h+4,&\text{if $b\le 3h-2$}\\
3h-4,&\text{otherwise}
\end{cases}
& -6 \\
\bottomrule
\end{tabular}
\end{table}

\subsection{Proof of \cref{lem:W}}\label[sec]{sec:W}
This section is to devoted to establish \cref{lem:W}.
In view of \cref{tab:def:W}, we need to deal with
$3+2+8+9=22$ cases.
We will handle them independently in 
\cref{prop:a=1:b=2,prop:a=2:b=4:n=6h+3,prop:a=4:b=8:b=11:a=7:b=14:n=6h+5,prop:a=1:b=3h+1:a=2:b=3h+3:n=6h+4,tab:W:6h+1,tab:W:6h+4,tab:W:6h+5}.
\begin{lemma}\label{lem:uuv}
Let $x,y,z$ be distinct elements in $\mathbb{Z}_n$.
Suppose that the edges $e_{i-1}^u$ and $e_i^u$
have the same edge metric representation with respect to 
the vertex triad $\{u_x,\,u_y,\,v_z\}$.
Then these edges have the same edge metric representation with respect to 
any one of the following triads:
\[
\{u_x,\,v_y,\,v_z\},\quad
\{v_x,\,u_y,\,v_z\},\quad
\{v_x,\,v_y,\,v_z\}.
\]
\end{lemma}
\begin{proof}
Recall from \eqref{At:Bt} that $A_t\subseteq B_t$.
Since $e_{i-1}^u$ and $e_i^u$ have the same distance from $u_x$,
we obtain $i\in A_x$. Thus $i\in B_x$ and the edges have the same distance from $v_x$. For the same reason, the edges have the same distance from $v_y$.
Hence they have the same edge metric representation with respect to 
each of the desired vertex sets. This completes the proof.
\end{proof}

Throughout this section, we suppose that $(a,b)\in W_n$.

\subsubsection{\cref{lem:W} is true if $r_n'=1$.}
Suppose that $r_n'=1$. Then 
\[
W_n=\{(1,\,3h-2),\
(2,\,3h),\
(5,\,3h+3)\}
\]
by definition.
In \cref{tab:W:6h+1}, 
we exhibit a pair of edges with the same edge metric representation with respect to 
a triad $R$ in $T(a,b)$, except $\{u_0,v_a,v_b\}$, $\{v_0,u_a,v_b\}$, and $\{v_0,v_a,v_b\}$. 
By \cref{lem:uuv},
this is enough to establish \cref{lem:W} for $r_n'=1$.
It is routine to check the truth of \cref{tab:W:6h+1} by \cref{thm:dist:ve}.
\begin{table}[htbp]
\centering
\caption{A pair of edges together with the same edge metric representation with respect to the vertex triad $\{\alpha_0,\beta_a,\gamma_b\}$, where $n=6h+1\ge 100$.}
\label{tab:W:6h+1}
\begin{tabular}{LLLL}
\toprule
&\multicolumn{3}{C}{(a,b)}\\
\cmidrule(r){2-4}
(\alpha,\beta,\gamma) & (1,\,3h-2) & (2,\,3h) & (5,\,3h+3) \\ 
\midrule
(u,u,u)
& \brk[c]1{e_{3h+3}^u,\,e_{3h+3}^s}
& \brk[c]1{e_2^u,\,e_2^s}
& \brk[c]1{e_5^u,\,e_5^s}\\[3pt]
& (h+1,\,h+2,\,4)
& (2,\,0,\,h+1)
& (4,\,0,\,h+1)\\ \addlinespace
(u,u,v)
& \brk[c]1{e_{3h}^u,\,e_{3h+1}^u}
& \brk[c]1{e_{3h+1}^u,\,e_{3h+2}^u}
& \brk[c]1{e_{3h+4}^u,\,e_{3h+5}^u}\\[3pt] 
& (h+2,\,h+2,\,2)
& (h+2,\,h+2,\,2)
& (h+1,\,h+2,\,2)\\ \addlinespace
(u,v,u)
& \brk[c]1{e_{3h-1}^u,\,e_{3h-1}^s}
& \brk[c]1{e_{3h+4}^u,\,e_{3h+5}^s}
& \brk[c]1{e_{3h+7}^u,\,e_{3h+8}^s}\\[3pt] 
& (h+2,\,h+1,\,1)
& (h+1,\,h+1,\,4)
& (h,\,h+1,\,4)\\ \addlinespace
(v,u,u)
& \brk[c]1{e_{3h+2}^u,\,e_{3h+3}^s}
& \brk[c]1{e_4^u,\,e_4^s}
& \brk[c]1{e_7^u,\,e_7^s}\\[3pt] 
& (h+1,\,h+2,\,4)
& (3,\,2,\,h+1)
& (4,\,2,\,h+1)\\ \addlinespace
(v,v,u)
& \brk[c]1{e_{-3}^u,\,e_{-2}^u}
& \brk[c]1{e_{-1}^u,\,e_0^u}
& \brk[c]1{e_2^u,\,e_3^u}\\[3pt] 
& (2,\,2,\,h+2)
& (1,\,2,\,h+2)
& (2,\,2,\,h+2)\\ 
\bottomrule
\end{tabular}
\end{table}

\subsubsection{\cref{lem:W} is true for $r_n'=3$.}
Suppose that $r_n'=3$. Then $W_n=\{(1,2),\,(2,4)\}$ by definition.

\begin{proposition}\label{prop:a=1:b=2}
Let $n\ge 100$. For any triad $R\in T(1,2)$, 
the edges $e_{-5}^u$ and $e_6^u$
have the same edge metric representation with respect to $R$.
\end{proposition}
\begin{proof}
By \cref{thm:dist:ve}, for any $t\in\{0,1,2\}$ and any $e\in\{e_{-5}^u,\,e_6^u\}$,
one may compute that $d(u_t,e)=4$ and $d(v_t,e)=3$.
\end{proof}

\begin{proposition}\label{prop:a=2:b=4:n=6h+3}
Let $n=6h+3\ge 100$ for some $h\in \mathbb{Z}$.
For any triad $R\in T(2,4)$, 
the edges $e_{3h+2}^u$ and $e_{3h+4}^u$
have the same edge metric representation with respect to $R$.
\end{proposition}
\begin{proof}
By \cref{thm:dist:ve}, for any $t\in\{0,2,4\}$ and any $e\in\{e_{3h+2}^u,\,e_{3h+4}^u\}$, one may compute that 
$d(u_t,\,e)=h+2$ and $d(v_t,\,e)=h+1$.
\end{proof}

By \cref{prop:a=1:b=2,prop:a=2:b=4:n=6h+3}, \cref{lem:W} is true for $r_n'=3$.

\subsubsection{\cref{lem:W} is true for $r_n'=4$}
Suppose that $r_n'=4$. By definition, we have
\[
W_n=\{(1,2),\
(1,\,3h-2),\
(1,\,3h+1),\
(2,3h),\
(2,\,3h+3),\
(4,\,3h+1),\
(5,\,3h+3),\
(8,\,3h+6)\}.
\]
\begin{proposition}\label{prop:a=1:b=3h+1:a=2:b=3h+3:n=6h+4}
Let $n=6h+4\ge 100$ for some $h\in \mathbb{Z}$. Then we have the following.
\begin{enumerate}
\item\label[itm]{itm:a=1:b=3h+1:n=6h+4}
For any triad $R\in T(1,\,3h+1)$, 
the edges $e_{-6}^u$ and $e_{5}^u$
have the same edge metric representation with respect to $R$.
\item\label[itm]{itm:a=2:b=3h+3:n=6h+4}
For any triad $R\in T(2,\,3h+3)$, 
the edges $e_{-5}^u$ and $e_{6}^u$
have the same edge metric representation with respect to $R$.
\end{enumerate}
\end{proposition}
\begin{proof}
By \cref{thm:dist:ve}, we have the following.
\begin{enumerate}
\item
For any edge $e\in\{e_{-6}^u,\,e_{5}^u\}$,
\[
d(u_0,\,e)=d(u_1,\,e)=4,\quad
d(u_{3h+1},\,e)=h+1,\quad
d(v_0,\,e)=d(v_1,\,e)=3,\quad\text{and}\quad
d(v_{3h+1},\,e)=h.
\]
\item
For any edge $e\in\{e_{-5}^u,\,e_{6}^u\}$,
\[
d(u_0,\,e)=d(u_2,\,e)=4,\quad
d(u_{3h+3},\,e)=h+1,\quad
d(v_0,\,e)=d(v_2,\,e)=3,\quad\text{and}\quad
d(v_{3h+3},\,e)=h.
\]
\end{enumerate}
This completes the proof.
\end{proof}

In view of \cref{prop:a=1:b=2,prop:a=1:b=3h+1:a=2:b=3h+3:n=6h+4},
we need to consider another 5 pairs $(a,b)\in W_n$.
By \cref{thm:dist:ve}, 
it is routine to check \cref{tab:W:6h+4}.
\begin{table}[htbp]
\centering
\caption{A pair of edges together with the same edge metric representation
with respect to the triad $\{\alpha_0,\beta_a,\gamma_b\}$, when $n=6h+4\ge 100$.}
\label{tab:W:6h+4}
\begin{tabular}{LLLLLL}
\toprule
&\multicolumn{5}{C}{(a,b)}\\
\cmidrule(r){2-6}
(\alpha,\beta,\gamma) & (1,\,3h-2) & (2,\,3h) & (4,\,3h+1) & (5,\,3h+3) & (8,\,3h+6) \\ 
\midrule
(u,u,u)
& \brk[c]1{e_{3h+3}^u,\,e_{3h+3}^s}
& \brk[c]1{e_{3h+5}^u,\,e_{3h+5}^s}
& \brk[c]1{e_{5}^u,\,e_{5}^s}
& \brk[c]1{e_{3h+8}^u,\,e_{3h+8}^s}
& \brk[c]1{e_{3h+11}^u,\,e_{3h+11}^s}\\[3pt]
& (h+2,h+3,4)
& (h+2,h+2,4)
& (4,1,h+1)
& (h+1,h+2,4)
& (h,h+2,4)\\ \addlinespace
(u,u,v)
& \brk[c]1{e_{3h-1}^u,\,e_{3h}^u}
& \brk[c]1{e_{3h+1}^u,\,e_{3h+2}^u}
& \brk[c]1{e_{3h+3}^u,\,e_{3h+4}^u}
& \brk[c]1{e_{3h+4}^u,\,e_{3h+5}^u}
& \brk[c]1{e_{3h+7}^u,\,e_{3h+8}^u}\\[3pt]
& (h+2,h+2,2)
& (h+3,h+2,2)
& (h+2,h+2,2)
& (h+2,h+2,2)
& (h+1,h+2,2)\\ \addlinespace
(u,v,u)
& \brk[c]1{e_{3h-1}^u,\,e_{3h-1}^s}
& \brk[c]1{e_{3h+5}^u,\,e_{3h+5}^s}
& \brk[c]1{e_{5}^u,\,e_{5}^s}
& \brk[c]1{e_{3h+8}^u,\,e_{3h+8}^s}
& \brk[c]1{e_{3h+11}^u,\,e_{3h+11}^s}\\[3pt]
& (h+2,h+1,1)
& (h+2,h+1,4)
& (4,2,h+1)
& (h+1,h+1,4)
& (h,h+1,4)\\ \addlinespace
(v,u,u)
& \brk[c]1{e_{3h+3}^u,\,e_{3h+3}^s}
& \brk[c]1{e_{4}^u,\,e_{4}^s}
& \brk[c]1{e_{3h+3}^u,\,e_{3h+3}^s}
& \brk[c]1{e_{7}^u,\,e_{7}^s}
& \brk[c]1{e_{10}^u,\,e_{10}^s}\\[3pt]
& (h+1,h+3,4)
& (3,2,h+1)
& (h+1,h+2,2)
& (4,2,h+1)
& (5,2,h+1)\\ \addlinespace
(v,v,u)
& \brk[c]1{e_{-3}^u,\,e_{-2}^u}
& \brk[c]1{e_{-1}^u,\,e_{0}^u}
& \brk[c]1{e_{1}^u,\,e_{2}^u}
& \brk[c]1{e_{2}^u,\,e_{3}^u}
& \brk[c]1{e_{5}^u,\,e_{6}^u}\\[3pt]
& (2,2,h+2)
& (1,2,h+2)
& (2,2,h+2)
& (2,2,h+2)
& (3,2,h+2)\\ 
\bottomrule
\end{tabular}
\end{table}
By \cref{lem:uuv}, we completes the proof of \cref{lem:W} for $r_n'=4$.

\subsubsection{\cref{lem:W} is true for $r_n'=5$}
Suppose that $r_n'=5$. By definition, we have
\[
W_n=\brk[c]{(1,2),\,(1,5),\,(1,8),\,(2,4),\,(2,7),\,(2,10),\,(4,8),\,(4,11),\,(7,14)}.
\]
\begin{proposition}\label{prop:a=4:b=8:b=11:a=7:b=14:n=6h+5}
Let $n=6h+5\ge 100$ for some $h\in \mathbb{Z}$.
Then we have the following.
\begin{enumerate}
\item
For any set $R\in T(4,8)\cup T(4,11)$, 
the edges $e_{3h+4}^v$ and $e_{3h+6}^v$
have the same edge metric representation with respect to $R$.
\item
For any set $R\in T(7,14)$, 
the edges $e_{3h+7}^v$ and $e_{3h+9}^v$
have the same edge metric representation with respect to $R$.
\end{enumerate}
\end{proposition}
\begin{proof}
By \cref{thm:dist:ve}, we have the following.
\begin{enumerate}
\item
For any edge $e\in\{e_{3h+4}^v,\,e_{3h+6}^v\}$,
\begin{align*}
d(u_0,\,e)&=d(u_4,\,e)=d(u_8,\,e)=d(v_{11},\,e)=h+1,\\
d(u_{11},\,e)&=d(v_4,\,e)=h\quad\text{and}\\
d(v_0,\,e)&=d(v_8,\,e)=h+2.
\end{align*}
\item
For any edge $e\in\{e_{3h+7}^v,\,e_{3h+9}^v\}$,
\[
d(u_0,\,e)=d(u_{14},\,e)=d(v_7,\,e)=h
\quad\text{and}\quad
d(u_7,\,e)=d(v_0,\,e)=d(v_{14},\,e)=h+1.
\]
\end{enumerate}
This completes the proof.
\end{proof}

In view of \cref{prop:a=1:b=2,prop:a=4:b=8:b=11:a=7:b=14:n=6h+5},
we need to consider another 5 pairs $(a,b)\in W_n$.
By \cref{thm:dist:ve}, 
it is routine to check \cref{tab:W:6h+5}.

\begin{table}[htbp]
\centering
\caption{A pair of edges together with the same edge metric representation with respect to the triad $\{\alpha_0,\beta_a,\gamma_b\}$, when $n=6h+5\ge 100$.}
\label{tab:W:6h+5}
\begin{tabular}{RLLL}
\toprule
(\alpha,\beta,\gamma) & (a,b)=(1,5) & (a,b)=(1,8) & (a,b)=(2,4) \\ 
\midrule
(u,u,u)
& \brk[c]1{e_{1}^u,\,e_{1}^s}
& \brk[c]1{e_{1}^u,\,e_{1}^s}
& \brk[c]1{e_{3h+2}^v,\,e_{3h+4}^v}\\[3pt]
& (1,\,0,\,3)
& (1,\,0,\,4)
& (h+1,\,h+1,\,h+1)\\ \addlinespace
(u,u,v)
& \brk[c]1{e_{3h+5}^u,\,e_{3h+3}^s}
& \brk[c]1{e_{3h+5}^u,\,e_{3h+3}^s}
& \brk[c]1{e_{3h+7}^u,\,e_{3h+5}^s}\\[3pt] 
& (h+2,\,h+2,\,h+1)
& (h+2,\,h+2,\,h)
& (h+1,\,h+2,\,h+2)\\ \addlinespace
(u,v,u)
& \brk[c]1{e_{2}^u,\,e_{2}^s}
& \brk[c]1{e_{5}^u,\,e_{5}^s}
& \brk[c]1{e_{3h+2}^v,\,e_{3h+4}^v}\\[3pt] 
& (2,\,2,\,2)
& (4,\,3,\,2)
& (h+1,\,h,\,h+1)\\ \addlinespace
(u,v,v)
& \brk[c]1{e_{3h+5}^u,\,e_{3h+3}^s}
& \brk[c]1{e_{3h+5}^u,\,e_{3h+3}^s}
& \brk[c]1{e_{3h+7}^u,\,e_{3h+5}^s} \\[3pt] 
& (h+2,\,h+1,\,h+1)
& (h+2,\,h+1,\,h)
& (h+1,\,h+1,\,h+2) \\ \addlinespace
(v,u,u)
& \brk[c]1{e_{1}^u,\,e_{1}^s}
& \brk[c]1{e_{1}^u,\,e_{1}^s}
& \brk[c]1{e_{4}^u,\,e_{4}^s} \\[3pt] 
& (2,\,0,\,3)
& (2,\,0,\,4)
& (3,\,2,\,0) \\ \addlinespace
(v,u,v)
& \brk[c]1{e_{3h+5}^u,\,e_{3h+3}^s}
& \brk[c]1{e_{3h+5}^u,\,e_{3h+3}^s}
& \brk[c]1{e_{3h+7}^u,\,e_{3h+5}^s} \\[3pt] 
& (h+1,\,h+2,\,h+1)
& (h+1,\,h+2,\,h)
& (h,\,h+2,\,h+2) \\ \addlinespace
(v,v,u)
& \brk[c]1{e_{3h+4}^u,\,e_{3h+7}^s}
& \brk[c]1{e_{3h+7}^u,\,e_{3h+10}^s}
& \brk[c]1{e_{3h+4}^u,\,e_{3h+3}^s} \\[3pt] 
& (h+1,\,h+2,\,h+2)
& (h,\,h+1,\,h+2)
& (h+1,\,h+2,\,h+2) \\ \addlinespace
(v,v,v)
& \brk[c]1{e_{3h+5}^u,\,e_{3h+3}^s}
& \brk[c]1{e_{3h+5}^u,\,e_{3h+3}^s}
& \brk[c]1{e_{3h+7}^u,\,e_{3h+5}^s} \\[3pt] 
& (h+1,\,h+1,\,h+1)
& (h+1,\,h+1,\,h)
& (h,\,h+1,\,h+2)\\
\bottomrule 
\end{tabular}

\begin{tabular}{LLLL}
\toprule
(\alpha,\beta,\gamma) & (a,b)=(2,7) & (a,b)=(2,10) \\ 
\midrule
(u,u,u)
& \brk[c]1{e_{3h+2}^v,\,e_{3h+4}^v}
& \brk[c]1{e_{3h+2}^v,\,e_{3h+4}^v}\\[3pt]
& (h+1,\,h+1,\,h)
& (h+1,\,h+1,\,h-1)\\ \addlinespace
(u,u,v)
& \brk[c]1{e_{3h+7}^u,\,e_{3h+5}^s}
& \brk[c]1{e_{3h+7}^u,\,e_{3h+5}^s}\\[3pt] 
& (h+1,\,h+2,\,h+1)
& (h+1,\,h+2,\,h)\\ \addlinespace
(u,v,u)
& \brk[c]1{e_{3h+2}^v,\,e_{3h+4}^v}
& \brk[c]1{e_{3h+2}^v,\,e_{3h+4}^v}\\[3pt]
& (h+1,\,h,\,h)
& (h+1,\,h,\,h-1)\\ \addlinespace
(u,v,v)
& \brk[c]1{e_{3h+7}^u,\,e_{3h+5}^s}
& \brk[c]1{e_{3h+7}^u,\,e_{3h+5}^s}\\[3pt] 
& (h+1,\,h+1,\,h+1)
& (h+1,\,h+1,\,h)\\ \addlinespace
(v,u,u)
& \brk[c]1{e_{1}^u,\,e_{2}^u}
& \brk[c]1{e_{1}^u,\,e_{2}^u}\\[3pt] 
& (2,\,0,\,4)
& (2,\,0,\,5)\\ 
\bottomrule
\end{tabular}
\end{table}

By \cref{lem:uuv}, we completes the proof of \cref{lem:W} for $r_n'=5$.

Now, by \cref{prop:S-W,lem:W}, we obtain \cref{prop:edim>=4}.

\section{Proof of \cref{prop:edim<=4}}\label[sec]{sec:edim<=4}

This section is devoted to proving that the edge dimension of 
the graph $P(n,3)$ is at most 4.
We proceed according to the residue $r_n'$.

\subsection{\cref{prop:edim<=4} is true if $r_n=0$.}
By definition, $r=1$ if $r_i=1$, and $r=0$ otherwise.
First of all, the edges whose distances from $u_0$ are less than $3$ are distinguishable;
see~\cref{tab:rn=0:d<3} 
for their edge metric representations with respect to the tetrad
$\{u_0,\,u_1,\,v_2,\,u_{\floor{n/2}-1}\}$,
where an asterisk~$*$ entry means that it is unnecessary to be computed
for the purpose of distinguishing the edge indicated by the row the entry lies in.
\begin{table}[htbp]
\centering
\caption{The edge metric representations of edges $e_i^l$ whose distances from $u_0$ are less than $3$, when $n=3q_n$.}
\label{tab:rn=0:d<3}
\begin{tabular}{@{} LLCLLL @{}}
\toprule
l 
& i 
&d\brk1{u_0,\,e_i^l}
&d\brk1{u_1,\,e_i^l}
&d\brk1{v_2,\,e_i^l}
&d\brk1{u_{\floor{n/2}-1},\,e_i^l} \\ 
\midrule
u 
& 0  & 0 & 0 & * & * \\ 
& 1  & 1 & 0 & 1 & * \\ 
& 2  & 2 & 1 & 1 & * \\ 
& -3 & 2 & 3 & 3 & \floor{q_n/2}+2 \\ 
& -2 & 1 & 2 & 2 & * \\ 
& -1 & 0 & 1 & 2 & * \\ 
\addlinespace
s
& 0  & 0 & 1 & 3 & * \\ 
& 1  & 1 & 0 & 2 & * \\ 
& 2  & 2 & 1 & 0 & * \\ 
& 3  & 2 & 2 & 2 & * \\ 
& -3 & 2 & 3 & 4 & \floor{q_n/2}+2 \\ 
& -2 & 2 & 2 & 3 & * \\ 
& -1 & 1 & 2 & 1 & * \\ 
\addlinespace
v
& 0  & 1 & 2 & 3 & * \\ 
& 1  & 2 & 1 & 3 & \ceil{q_n/2} \\ 
& 3  & 2 & 3 & 3 & \le\floor{q_n/2} \\
& -6 & 2 & 3 & 4 & \floor{q_n/2} \\ 
& -4 & 2 & 3 & 1 & * \\ 
& -3 & 1 & 2 & 4 & * \\ 
& -2 & 2 & 1 & 3 & \ceil{q_n/2}+1 \\ 
& -1 & 2 & 2 & 0 & * \\ 
\bottomrule
\end{tabular}
\end{table}

While most entries in \cref{tab:rn=0:d<3} can be obtained directly
from \cref{thm:dist:ve}, we explain
the $u_{\floor{n/2}-1}$-coordinate of the row $e_3^v$,
which is derived in the following way.
Consider 
\[
i=\ceil{n/2}+4=\begin{cases}
3h+4,&\text{if $n=6h$};\\
3h+6,&\text{if $n=6h+3$}.
\end{cases}
\] 
It follows that 
\[
(q_i,\,r_i,\,r)=\begin{cases}
(h+1,\,1,\,1),&\text{if $n=6h$};\\
(h+2,\,0,\,0),&\text{if $n=6h+3$}.
\end{cases}
\]
Therefore,
\[
d\brk1{u_{\floor{n/2}-1},\,e_3^v}
=d\brk1{u_0,\,e_{\ceil{n/2}+4}^v}
=\min(q_i+r_i+1,\,q_n-q_i+r)
=\begin{cases}
h,&\text{if $n=6h$}\\
h-1,&\text{if $n=6h+3$}
\end{cases}
\]
is at most $\floor{q_n/2}$.

Let $d\ge 3$ be an integer.
By using \cref{thm:dist:ve}, 
one may solve out the set of edges whose distance from $u_0$ is $d$,
under some lower bound conditions on $q_n$ and~$d$; see~\cref{tab:edge:rn=0}.
The lower bound of $d$ works for all integers $n\ge 18$,
since it is derived by requiring that distance from the edge $e_i^l$ to $u_0$
is computed by using the second expression in each of the first three formulas
in \cref{thm:dist:ve}.
\begin{table}
\centering
\caption{The edges $e_i^l$ whose distances from $u_0$ are $d$, with the corresponding lower bounds of $q_n$ and $d$, when $n=3q_n$.}
\label{tab:edge:rn=0}
\begin{tabular}{@{} LLCC @{}}
\addlinespace
\toprule
l 
& i
& \text{lower bound of $q_n$}
& \text{lower bound of $d$} \\ 
\midrule
u 
& 3d-6 & 2d-4 & 3 \\ 
& 3d-7 & 2d-4 & 4 \\ 
& 3d-8 & 2d-5 & 4 \\ 
& 5-3d & 2d-4 & 3 \\ 
& 6-3d & 2d-4 & 4 \\ 
& 7-3d & 2d-5 & 4 \\ 
\addlinespace
s
& \pm(3d-3) & 2d-2 & 3 \\ 
& \pm(3d-5) & 2d-4 & 3 \\ 
& \pm(3d-7) & 2d-4 & 4 \\ 
\addlinespace
v
& 3d-3 & 2d-1 & 3 \\ 
& 3d-5 & 2d-3 & 3 \\ 
& 3d-7 & 2d-3 & 3 \\ 
& -3d  & 2d-1 & 3 \\ 
& 2-3d & 2d-3 & 3 \\ 
& 4-3d & 2d-3 & 3 \\ 
\bottomrule
\end{tabular}
\end{table}
It remains to show that the edge metric representations are distinct.

From \cref{tab:edge:rn=0}, 
we see that the minimum lower bound of $q_n$ is $2d-5$,
which is attained by the edges $e_{3d-8}^u$ and $e_{7-3d}^u$.
In fact, these two edges coincide with each other, because
\[
(3d-8)-(7-3d)=6d-15=3(2d-5)=3q_n=n.
\]
Using the same idea, 
we see 4 edges for which the lower bound of $q_n$ is $2d-4$:
\[
e_{3d-6}^u=e_{6-3d}^u,\quad
e_{3d-7}^u=e_{5-3d}^u,\quad
e_{3d-5}^s=e_{7-3d}^s,\quad\text{and}\quad
e_{3d-7}^s=e_{5-3d}^s.
\]
They have distances $d-1$, $d-2$, $d$, and $d-3$ from $v_2$ respectively,
and thus distinguishable.

Below we can suppose that $q_n\ge 2d-3$.
By \cref{thm:dist:ve}, we compute the metric representations of edges in \cref{tab:edge:rn=0} with respect to the remaining vertex triple $(u_1,v_2,u_{\floor{n/2}-1})$; see~\cref{tab:rn=0:d},
where $I_{even}(q_n)$ equals $1$ if $q_n$ is odd and $0$ if $q_n$ is even. 
\begin{table}
\centering
\caption{The metric representations of the edges whose distances from $u_0$ are the same number $d$, when $n=3q_n\ge 3(2d-3)$.}
\label{tab:rn=0:d}
\begin{tabular}{@{} LLLLLL @{}}
\addlinespace
\toprule
l & i & q_n & d(u_1,\,e_i^l) & d(v_2,\,e_i^l) & d(u_{\floor{n/2}-1},\,e_i^l)\\ 
\midrule
u
& 3d-6 
& \ge 2d-3
& 2\ (d=3) 
& d-1 
& \le \floor{q_n/2}-d+4 \\ 
& 
& 
& d\ (d\ge 4)
& 
& \\ 
& 3d-7 & 2d-3     & d   & d-2 & 0 \\ 
&      & 2d-2     &     &     & 2 \\ 
&      & \ge 2d-1 &     &     & * \\ 
& 3d-8 & \ge 2d-3 & d-1 & d-2 & * \\ 
& 5-3d & 2d-3     & d+1 & d-1 & * \\  
&      & \ge 2d-2 &     & d   & \ceil{q_n/2}-d+4\\ 
& 6-3d & \ge 2d-3 & d   & d   & * \\ 
& 7-3d & \ge 2d-3 & d   & d-1 & \ceil{q_n/2}-d+5\\ 
\addlinespace
s 
& 3d-3 & \ge 2d-2 & d+1 & d   & \le \floor{q_n/2}-d+3 \\ 
& 3d-5 & \ge 2d-3 & d-1 & d   & \le \ceil{q_n/2}-d+3 \\ 
& 3d-7 & 2d-3     & d-1 & d-3 & 1 \\ 
&      & \ge 2d-2  &     &     & \ceil{q_n/2}-d+3 \\ 
& 3-3d & 2d-2     & d+1 & d   & * \\ 
&      & 2d-1     &     & d+1 & 2 \\
&      & \ge 2d   &     & d+2 & \floor{q_n/2}-d+3 \\
& 5-3d & 2d-3     & d   & d-2 & 2 \\ 
&      & \ge 2d-2 & d+1 & d-1 & \floor{q_n/2}-d+3+2 I_{even}(q_n) \\ 
& 7-3d & \ge 2d-3 & d-1 & d   & \floor{q_n/2}-d+5\\ 
\addlinespace
v 
& 3d-3 & 2d-1     & d+1 & d+1 & 1 \\ 
&      & \ge 2d   &     &     & * \\ 
& 3d-5 & 2d-3     & d-1 & d+1 & * \\ 
&      & 2d-2     &     &     & 2 \\ 
&      & \ge 2d-1 &     &     & \ceil{q_n/2}-d+2 \\ 
& 3d-7 & 2d-3     & d-1 & d-3 & 2 \\ 
&      & \ge 2d-2 &     &     & \ceil{q_n/2}-d+2 \\ 
& -3d  & 2d-1     & d+1 & d+1 & * \\ 
&      & 2d       &     & d+2 & 2 \\ 
&      & \ge 2d+1 &     & d+3 & * \\ 
& 2-3d & 2d-3     & d-1 & d-3 & * \\ 
&      & 2d-2     & d   & d-2 & 1 \\ 
&      & \ge 2d-1 & d+1 & d-1 & \floor{q_n/2}-d+2+2 I_{even}(q_n)\\ 
& 4-3d & \ge 2d-3 & d-1 & d+1 & \floor{q_n/2}-d+4 \\
\bottomrule
\end{tabular}
\end{table}
Note that 
\begin{align*}
&3d-5\equiv 4-3d     \pmod{n}\quad\text{when $q_n=2d-3$},\\
&3d-7\equiv 2-3d     \pmod{n}\quad\text{when $q_n=2d-3$},\\
&3d-3\equiv 3-3d     \pmod{n}\quad\text{when $q_n=2d-2$},\\
&3d-3\equiv -3d\quad \pmod{n}\quad\text{when $q_n=2d-1$}.
\end{align*}
From~\cref{tab:rn=0:d},
we see that every edge is recognizable, as desired.

\subsection{\cref{prop:edim<=4} is true if $r_n'=1$.}
By definition, $r=\abs{1-r_i}$.
Let $d\ge 0$ be an integer.
For $d\le2$, all entries except the non-asterisks in the last column in \cref{tab:rn=0:d<3} keep invariant.
The non-asterisks are listed in~\cref{tab:rn'=1:d<3},
From which we see that the edges whose distance $d$ 
from $u_0$ is at most 2 are distinguishable by the quadruple $(u_0,\,u_1,\,v_2,\,u_{\floor{n/2}-1})$.

\begin{table}
\caption{The edge metric representations for $d\le 2$
that contains the vertex $u_{\floor{n/2}-1}$ in an edge resolving set, 
when $n=6h+1\ge 19$.}
\label{tab:rn'=1:d<3}
\centering
\begin{tabular}{@{} LLCCCC @{}}
\toprule
l
& i
& d\brk1{u_0,\,e_i^l}
& d\brk1{u_1,\,e_i^l}
& d\brk1{v_2,\,e_i^l}
& d\brk1{u_{\floor{n/2}-1},\,e_i^l} \\ 
\midrule
u & -3 & 2 & 3 & 3 & h+2 \\ 
s & -3 & 2 & 3 & 4 & h+2 \\ 
v & 1  & 2 & 1 & 3 & h \\ 
  & 3  & 2 & 3 & 3 & h \\ 
  & -6 & 2 & 3 & 4 & h+1 \\ 
  & -2 & 2 & 1 & 3 & h+1 \\ 
\bottomrule
\end{tabular}
\end{table}

Let $d\ge 3$.
Note that 
\[
3d-3\equiv 2-3d
\quad\text{and}\quad
3d-5\equiv -3d\pmod{n}\quad\text{when $d=h+1$}.
\]
We list the edges whose distance from $u_0$ is~$d$,
with their metric representations with respect to the remaining vertex triple $(u_1,v_2,u_{\floor{n/2}-1})$ as in~\cref{tab:6h+1:d},
From which we see that all edges are distinguishable by the tetrad 
$\{u_0,\,u_1,\,v_2,\,u_{\floor{n/2}-1}\}$, as desired.
\begin{table}
\centering
\caption{The metric representations of the edges $e_i^l$ whose distances from $u_0$ are $d$, when $n=6h+1$.}
\label{tab:6h+1:d}
\begin{tabular}{@{} LLLLLL @{}}
\addlinespace
\toprule
l
& i
& d
& d\brk1{u_1,\,e_i^l}
& d\brk1{v_2,\,e_i^l}
& d\brk1{u_{\floor{n/2}-1},\,e_i^l} \\
\midrule
u
& 3d-6 
& h+2
& d\ (d\ge 4)
& d-1 
& h-d+3 \\ 
& 
& h+1
& 
& 
& h-d+2 \\ 
& 
& [3,\,h] 
& 2\ (d=3)
& 
& h-d+4\\ 
& 
& 
& d\ (d\ge 4)
& 
& \\ 
& 3d-7 & [4,\,h+2] & d   & d-2 & * \\ 
& 3d-8 & h+2       & d-1 & d-2 & h-d+2 \\ 
&      & [4,\,h+1] &     &     & * \\ 
& 5-3d & h+2       & d   & d-1 & * \\ 
&      & [3,\,h+1] & d+1 & d   & h-d+5 \\
& 6-3d & h+2       & d   & d-1 & h-d+4 \\
&      & [4,\,h+1] &     & d   & * \\
& 7-3d & [4,\,h+2] & d   & d-1 & h-d+5 \\
\addlinespace
s 
& 3d-3 & [3,\,h+1] & d+1 & d   & \le h-d+3 \\ 
& 3d-5 & [3,\,h+1] & d-1 & d   & \le h-d+3 \\ 
& 3d-7 & h+2       & d-1 & d-3 & h-d+2 \\
&      & [4,\,h+1] &     &     & h-d+3 \\
& 3-3d & h+1       & d   & d+1 & * \\
&      & [3,\,h]   & d+1 & d+2 & * \\
& 5-3d & [3,\,h+1] & d+1 & d-1 & h-d+4\\ 
& 7-3d & h+2       & d-1 & d-2 & h-d+4 \\ 
&      & [4,\,h+1] &     & d   & h-d+4 \\ 
\addlinespace
v 
& 3d-3 & h+1       & d+1 & d-1 & * \\
&      & [3,\,h]   &     & d+1 & * \\ 
& 3d-5 & h+1       & d-1 & d+1 & h-d+3 \\
&      & [3,\,h]   &     &     & h-d+2 \\ 
& 3d-7 & h+2       & d-1 & d-3 & h-d+3 \\ 
&      & [3,\,h+1] &     &     & h-d+2 \\ 
& -3d  & h+1       & d-1 & d+1 & * \\
&      & [3,\,h]   & d+1 & d+3 & * \\
& 2-3d & [3,\,h+1] & d+1 & d-1 & h-d+3 \\
& 4-3d & h+2       & d-1 & d-3 & * \\
&      & h+1       &     & d-1 & * \\
&      & [3,\,h]   &     & d+1 & h-d+3 \\ 
\bottomrule
\end{tabular}
\end{table}

\subsection{\cref{prop:edim<=4} is true if $r_n'=2$.}
From definition, $r=\abs{1-r_i}$.
For $d\le2$, we compute the edge metric representations as in~\cref{tab:6h+2:d<3},
in which
the columns for $d(u_0,e_i^l)$ and $d(u_1,e_i^l)$ are same to those in \cref{tab:rn=0:d<3}.
\begin{table}
\centering
\caption{The metric representations of edges $e_i^l$ whose distances from $u_0$ are at most $2$, when $n=6h+2$.}
\label{tab:6h+2:d<3}
\begin{tabular}{@{} LLCLLL @{}}
\toprule
l 
& i 
&d\brk1{u_0,\,e_i^l}
&d\brk1{u_1,\,e_i^l}
& d\brk1{u_{n/2-3},\,e_i^l}
& d\brk1{v_{n/2-2},\,e_i^l} \\ 
\midrule
u 
& 0  & 0 & 0 & *   & * \\ 
& 1  & 1 & 0 & h+1 & * \\ 
& 2  & 2 & 1 & h+1 & h \\ 
& -3 & 2 & 3 & h+2 & h+1 \\ 
& -2 & 1 & 2 & h+2 & h+1 \\ 
& -1 & 0 & 1 & h+2 & * \\ 
\addlinespace
s
& 0  & 0 & 1 & h+1 & * \\ 
& 1  & 1 & 0 & h   & * \\ 
& 2  & 2 & 1 & h+1 & h-1 \\ 
& 3  & 2 & 2 & h   & * \\ 
& -3 & 2 & 3 & h+2 & h \\ 
& -2 & 2 & 2 & h+1 & h+2 \\ 
& -1 & 1 & 2 & h+2 & h \\ 
\addlinespace
v
& 0  & 1 & 2 & h   & * \\ 
& 1  & 2 & 1 & h-1 & * \\ 
& 3  & 2 & 3 & h-1 & * \\
& -6 & 2 & 3 & h+1 & h-1 \\ 
& -4 & 2 & 3 & h+1 & h \\ 
& -3 & 1 & 2 & h+1 & * \\ 
& -2 & 2 & 1 & h   & * \\ 
& -1 & 2 & 2 & h+1 & h-1 \\ 
\bottomrule
\end{tabular}
\end{table}
Let $d\ge 3$.
Note that 
\begin{align*}
3d-5\equiv 5-3d\quad\text{and}\quad 3d-7\equiv 3-3d\pmod{n}\quad\text{when $d=h+2$}.
\end{align*}
We list the edges whose distance from $u_0$ is~$d$,
with their metric representations with respect to the remaining vertex triple 
in~\cref{tab:6h+2:d}, 
from which we see that all edges are distinguishable, as desired.
\begin{table}
\centering
\caption{The metric representations of the edges $e_i^l$ whose distances from $u_0$ are $d$, when $n=6h+2$.}
\label{tab:6h+2:d}
\begin{tabular}{@{} LLLLLL @{}}
\addlinespace
\toprule
l 
& i 
& d
& d\brk1{u_1,\,e_i^l}
& d\brk1{u_{n/2-3},\,e_i^l}
& d\brk1{v_{n/2-2},\,e_i^l} \\ 
\midrule
u
& 3d-6 
& h+2
& d\ (d\ge 4)
& h-d+4
& * \\
&  
& h+1
& 
& h-d+1
& * \\
& 
& [3,\,h]
& 2\ (d=3)
& h-d+3
& * \\ 
& 
& 
& d\ (d\ge 4)
& 
& * \\   
& 3d-7 & h+2       & d   & h-d+3 & * \\ 
&      & h+1       &     & h-d+2 & * \\ 
&      & [4,\,h]   &     & h-d+4 & * \\ 
& 3d-8 & h+2       & d-1 & h-d+2 & * \\
&      & h+1       &     & h-d+3 & h-d+3 \\
&      & [4,\,h]   &     & h-d+4 & \\ 
& 5-3d & [3,\,h+2]   & \ge d & h-d+5 & * \\
& 6-3d & [4,\,h+2] & d   & h-d+6 & h-d+4 \\
& 7-3d & [4,\,h+2] & d   & h-d+6 &  h-d+5 \\
\addlinespace
s 
& 3d-3 & h+2       & d-1 & h-d+6 & h-d+3 \\ 
&      & h+1       & d+1 & h-d+3 & * \\ 
&      & h         &     & h-d+1 & * \\ 
&      & [3,\,h-1] &     & h-d+2 & * \\ 
& 3d-5 & h+2       & d-1 & h-d+4 & * \\ 
&      & h+1       &     & h-d+1 & * \\ 
&      & [3,\,h]   &     & h-d+2 & * \\ 
& 3d-7 & h+2       & d-1 & h-d+3 & * \\
&      & h+1       &     & h-d+3 & h-d+2 \\	
&      & [4,\,h]   &     & h-d+4 & h-d+2 \\
& 3-3d & h+2       & d-1 & h-d+3 & * \\
&      & [3,\,h+1] & d+1 & h-d+4 & h-d+2 \\
& 5-3d & h+2       & d-1 & h-d+4 & * \\
&      & [3,\,h+1] & d+1 &       & h-d+4\ (d=3) \\ 
&      &           &     &       & h-d+5\ (d\ge4)\\
& 7-3d & [4,\,h+2]   & d-1 & h-d+6 & h-d+5	 \\ 
\addlinespace
v 
& 3d-3 & h+1       & d   & h-d+4 & * \\ 
&      & h         & d+1 & h-d+2 & * \\ 
&      & [3,\,h-1] &     & h-d+1 & * \\ 
& 3d-5 & h+1       & d-1 & h-d+2 & * \\ 
&      & [3,\,h]   &     & h-d+1 & * \\ 
& 3d-7 & [3,\,h+1] & d-1 & h-d+3 & h-d+1  \\ 
& -3d  & h+1       & d   & h-d+3 & h-d+1 \\
&      & [3,\,h]   & d+1 &       &  \\
& 2-3d & h+1       & d   & h-d+3 & \ge h-d+4\\
&      & [3,\,h]   & d+1 &       &  \\
& 4-3d & [3,\,h+1] & d-1 & h-d+4\ (d=3) & * \\
&      &           &     & h-d+5\ (d\ge 4) & * \\
\bottomrule
\end{tabular}
\end{table}

\subsection{\cref{prop:edim<=4} is true if $r_n'=4$.}
By definition, $r=\abs{1-r_i}$. 
For $d\le2$, all entries except the non-asterisks in the last column in \cref{tab:rn=0:d<3} keep invariant.
The non-asterisks are listed in~\cref{tab:6h+4:d<3}, from which we 
see that the edges whose disctances from $u_0$ are at most 2
are distinguishable by $(u_0,\,u_1,\,v_2,\,u_{n/2+3})$.

\begin{table}
\caption{The metric representations for edges with $d\le 2$ that need the vertex $u_{n/2+3}$ as a resolving set, when $n=6h+4$.}
\label{tab:6h+4:d<3}
\centering
\begin{tabular}{@{} CCCCCC @{}}
\toprule
l
& i
& d\brk{u_0,\,e_i^l}
& d\brk{u_1,\,e_i^l}
& d\brk{v_2,\,e_i^l}
& d\brk{u_{n/2+3},\,e_i^l} \\ 
\midrule
u & -3 & 2 & 3 & 3 & h+1 \\ 
s & -3 & 2 & 3 & 4 & h+1 \\ 
v & 1  & 2 & 1 & 3 & h+1 \\ 
  & 3  & 2 & 3 & 3 & h+2 \\ 
  & -6 & 2 & 3 & 4 & h \\ 
  & -2 & 2 & 1 & 3 & h \\ 
\bottomrule
\end{tabular}
\end{table}

Let $d\ge 3$.
Note that 
\begin{align*}
&3d-7\equiv7-3d\quad\text{and}\quad 3d-8\equiv6-3d\pmod{n}\quad\text{when $d=h+3$},\\
&3d-3\equiv5-3d\quad\text{and}\quad 3d-5\equiv3-3d\pmod{n}\quad\text{when $d=h+2$}.
\end{align*}
We list the edges whose distance from $u_0$ is~$d$,
with their metric representations with respect to the remaining vertex triple 
as in~\cref{tab:6h+4:d},
from which we see that all edges are distinguishable, as desired.
\begin{table}
\centering
\caption{The metric representations of the edges $e_i^l$ whose distances from $u_0$ are $d$, when $n=6h+4$.}
\label{tab:6h+4:d}
\begin{tabular}{@{} LLLLLL @{}}
\addlinespace
\toprule
l & i & d & d\brk{u_1,\,e_i^l} & d\brk{v_2,\,e_i^l} & d\brk{u_{n/2+3},\,e_i^l} \\
\midrule
u
& 3d-6 
& [3,\,h+2] 
& 2\ (d=3)
& d-1
& h-d+6\\ 
&  
& 
& d\ (d\ge 4)
& 
& \\ 
& 3d-7 & [4,\,h+3] & d   & d-2 & * \\ 
& 3d-8 & [4,\,h+3] & d-1 & d-2 & h-d+6 \\ 
& 5-3d & h+2       & d+1 & d   & h-d+3 \\
&      & h+1       &     &     & h-d+2 \\
&      & [3,\,h]   &     &     & h-d+4 \\
& 6-3d & h+3       & d-1 & d-2 & * \\
&      & [4,\,h+2] & d   & d   & * \\
& 7-3d & h+3       & d   & d-2 & * \\
&      & h+2       &     & d-1 & h-d+2 \\
&      & [4,\,h+1] &     &     & h-d+4 \\
\addlinespace
s 
& 3d-3 & h+2       & d+1 & d-1 & * \\ 
&      & [3,\,h+1] &     & d   & h-d+5 \\ 
& 3d-5 & [3,\,h+2] & d-1 & d   & h-d+5 \\ 
& 3d-7 & [4,\,h+3] & d-1 & d-3 & h-d+5 \\
& 3-3d & h+2       & d-1 & d   & * \\
&      & [3,\,h+1] & d+1 & d+2 & * \\
& 5-3d & h+2       & d+1 & d-1 & * \\ 
&      & h+1       &     &     & h-d+2 \\ 
&      & [3,\,h]   &     &     & h-d+3 \\ 
& 7-3d & h+3       & d-1 & d-3 & * \\ 
&      & h+2       &     & d-1 & * \\ 
&      & [4,\,h+1] &     & d   & h-d+3 \\ 
\addlinespace
v 
& 3d-3 & h+1       & d+1 & d   & h-d+4 \\ 
&      & [3,\,h]   &     & d+1 & * \\ 
& 3d-5 & [3,\,h+1] & d-1 & d+1 & h-d+4 \\ 
& 3d-7 & [3,\,h+2] & d-1 & d-3 & h-d+4 \\ 
& -3d  & h+1       & d   & d+2 & * \\
&      & [3,\,h]   & d+1 & d+3 & * \\
& 2-3d & h+1       & d+1 & d-1 & h-d+3 \\
&      & [3,\,h]   &     &     & h-d+2 \\
& 4-3d & h+2       & d-1 & d-2 & h-d+3 \\ 
&      & h+1       &     & d   & h-d+2 \\ 
&      & [4,\,h]   &     & d+1 & h-d+2 \\ 
\bottomrule
\end{tabular}
\end{table}

\subsection{\cref{prop:edim<=4} is true if $r_n'=5$.}
Note that in this case the parameter $r=\abs{1-r_i}$.
For $d\le2$, we compute the edge metric representations as in~\cref{tab:6h+5:d<3}.
All entries for the first four columns in \cref{tab:rn=0:d<3} keep invariant.
\begin{table}
\centering
\caption{The metric representations of edges $e_i^l$ whose distances from $u_0$ are at most $2$, when $n=6h+5$.}
\label{tab:6h+5:d<3}
\begin{tabular}{@{} LLCLLL @{}}
\addlinespace
\toprule
l 
& i 
& d\brk{u_0,\,e_i^l}
& d\brk{u_1,\,e_i^l}
& d\brk{u_{(n-3)/2},\,e_i^l}
& d\brk{v_{(n-1)/2},\,e_i^l} \\ 
\midrule
u 
& 0  & 0 & 0 & *   & * \\ 
& 1  & 1 & 0 & h+2 & * \\ 
& 2  & 2 & 1 & h+2 & h+1 \\ 
& -3 & 2 & 3 & h+3 & * \\ 
& -2 & 1 & 2 & h+3 & * \\ 
& -1 & 0 & 1 & h+3 & * \\ 
\addlinespace
s
& 0  & 0 & 1 & h+2 & * \\ 
& 1  & 1 & 0 & h+1 & * \\ 
& 2  & 2 & 1 & h+2 & h \\ 
& 3  & 2 & 2 & h+1 & * \\ 
& -3 & 2 & 3 & h+2 & * \\ 
& -2 & 2 & 2 & h+2 & h+2 \\ 
& -1 & 1 & 2 & h+2 & h+1 \\ 
\addlinespace
v
& 0  & 1 & 2 & h+1 & * \\ 
& 1  & 2 & 1 & h   & * \\ 
& 3  & 2 & 3 & h   & * \\
& -6 & 2 & 3 & h+1 & h-1 \\ 
& -4 & 2 & 3 & h+1 & h+1 \\ 
& -3 & 1 & 2 & h+2 & h \\ 
& -2 & 2 & 1 & h+1 & * \\ 
& -1 & 2 & 2 & h+2 & h \\ 
\bottomrule
\end{tabular}
\end{table}

Let $d\ge 3$. We note that $d=h+3$ happens only when $l=u$.
Since $n=6h+5=6d-13$, among the 18 edges there are only three distinct edges have the same distance $d$ from $u_0$, that is, 
\[
e_{3d-6}^u=e_{7-3d}^u,\quad
e_{3d-7}^u=e_{6-3d}^u,\quad\text{and}\quad
e_{3d-8}^u=e_{5-3d}^u.
\]
Their distances from $u_{(n-3)/2}$ are respectively $2$, $1$, and $0$.
Thus these three edges are distinguishable.
For $d\le h+2$, we compute out~\cref{tab:6h+5:d},
from which we see that all edges are distinguishable. 
\begin{table}
\centering
\caption{The metric representations of the edges $e_i^l$ whose distance from $u_0$ are $d$, when $n=6h+5$ and $d\le h+2$.}
\label{tab:6h+5:d}
\begin{tabular}{@{} LLLLLL @{}}
\addlinespace
\toprule
l 
& i 
& d
& d\brk{u_1,\,e_i^l}
& d\brk{u_{(n-3)/2},\,e_i^l}
& d\brk{v_{(n-1)/2},\,e_i^l} \\ 
\midrule
u  
& 3d-6
& h+2
& d\ (d\ge 4)
& h-d+2=0
& * \\   
& 
& [3,\,h+1]
& 2\ (d=3)
& h-d+4
& h-d+4 \\   
& 
& 
& d\ (d\ge 4)
& 
& * \\   
& 3d-7 & h+2       & d   & h-d+3=1 & h-d+4 \\ 
&      & [4,\,h+1] &     & h-d+5   & * \\ 
& 3d-8 & h+2       & d-1 & h-d+4=2 & h-d+4 \\
&      & [4,\,h+1] &     & h-d+5   &  \\ 
& 5-3d & [3,\,h+2] & d+1 & h-d+5   & * \\
& 6-3d & [4,\,h+2] & d   & h-d+6   & h-d+4 \\
& 7-3d & [4,\,h+2] & d   & h-d+6   & h-d+5 \\
\addlinespace
s 
& 3d-3 & h+2       & d   & h-d+4=2 & h-d+4 \\ 
&      & h+1       & d+1 & h-d+2=1 & * \\ 
&      & [3,\,h]   &     & h-d+3   & * \\ 
& 3d-5 & h+2       & d-1 & h-d+2=0 & * \\ 
&      & [3,\,h+1] &     & h-d+3   & * \\ 
& 3d-7 & h+2       & d-1 & h-d+4=2 & h-d+3 \\
&      & [4,\,h+1] &     & h-d+5   &  \\	
& 3-3d & h+2       & d   & h-d+3=1 & h-d+2 \\
&      & [3,\,h+1] & d+1 & h-d+4   &  \\
& 5-3d & h+2       & d   & h-d+4   & h-d+5 \\
&      & [3,\,h+1] & d+1 &         &  \\ 
& 7-3d & [3,\,h+2] & d-1 & h-d+5\ (d=3)& h-d+5 \\ 
&      &           &     & h-d+6\ (d\ge 4)& * \\ 
\addlinespace
v 
& 3d-3 & h+2       & d-1           & h-d+6=4 & * \\ 
&      & h+1       & d+1           & h-d+4=3 & h-d+4 \\ 
&      & [3,\,h]   &               & h-d+2   & * \\ 
& 3d-5 & h+2       & d-1           & h-d+3=1 & * \\ 
&      & [3,\,h+1] &               & h-d+2   & * \\ 
& 3d-7 & [3,\,h+2] & 2\ (d=3)      & h-d+4   & h-d+2 \\ 
&      &           & d-1\ (d\ge 4) &         &  \\ 
& -3d  & h+2       & d             & h-d+5=3 & * \\
&      & [3,\,h+1] &               & h-d+3   & h-d+1 \\
& 2-3d & [3,\,h+2] & d             & h-d+3   & h-d+5 \\
& 4-3d & [3,\,h+2] & d-2           & h-d+5   & * \\
\bottomrule
\end{tabular}
\end{table}

Now,
we obtain a complete proof of \cref{prop:edim<=4}.
Recall that \cref{thm:edim=4} is true for $n<100$.
By \cref{prop:edim>=4,prop:edim<=4}, 
we complete the proof of \cref{thm:edim=4}.

\section{Concluding remarks}
The graphs $P(7,3)$ and $P(7,2)$ are isomorphic to each other.
By a result of Filipovi\'c et al.~\cite{FKK19}, 
we know that $P(7,2)$ has edge dimension 4.
It is routine to check that the M\"obius-Kantor graph $P(8,3)$ has edge dimensions 4,
and both the graphs $P(9,3)$ and $P(10,3)$ have edge dimension 3.
We think it a challenge to find out the edge dimension of $P(n,k)$ for $k\ge 4$.

\end{document}